\documentclass[a4paper,12pt]{amsart}
\usepackage{amsthm,amsmath,amssymb}
\usepackage{tikz-cd}
\usepackage{mathscinet}
\usepackage{pdfsync}
\usepackage[upint,notext]{stix2}
\usepackage{ascmac}
\usepackage{dashbox}
\usepackage{paralist}
\newenvironment{enumerate*}{\vskip.0ex
\begin{inparaenum}[a)\hspace{.2em}]
}{\end{inparaenum}\vspace{1ex}\par\noindent}
\def\vitem{\par\vskip1ex\noindent\hskip1em\item}
\def\hitem{\hskip1em\item}
\usepackage{xcolor}
\definecolor{slateblue}{rgb}{0.28,0.24,0.55}
\definecolor{forestgreen}{rgb}{0.13, 0.55, 0.13}
\definecolor{darkbrown}{rgb}{0.65, 0.16, 0.16}
\date{1st December 2022} 	
\date{28th December 2022}	
\date{20th March 2023}		
\date{9th September 2023}	
\date{3th February 2024}	
\date{5th January 2025}		
\date{\today}
\allowdisplaybreaks[4]
\makeatletter
\def\claimname{Claim}
\def\preclaimword{}

\def\@postclaim[#1]{\bf #1}
\def\@claim[#1]{\preclaimword {\bf #1.} \@ifnextchar[{\@postclaim}{}}
{\par\vskip1.5ex\noindent\@ifnextchar[{\@claim}{\bf\claimname~}}%
{\unskip\nobreak\hfill\par\vskip1.5ex}
\def\fracinline#1/#2{\mbox{\raise0.4ex\hbox{\footnotesize$#1$}{\hskip-.1em\mbox{\raise0.125ex\hbox{\small$/$}}\hskip-.1em}\raise-0.25ex\hbox{\footnotesize$#2$}}}
\def\fracinlines#1/#2{\mbox{\raise0.125ex\hbox{\tiny$#1$}{\hskip-.075em\mbox{\raise0.125ex\hbox{\tiny$/$}}\hskip-.1em}\raise-0.15ex\hbox{\tiny$#2$}}}
\makeatother

\makeatletter
\def\proofname{Proof}
\def\preproofword{\it Proof of }
\def\postproofword{\it :~}
\def\@pf[#1]{\preproofword {\it #1} \postproofword~}
\newenvironment{Proof}%
{\par\noindent\@ifnextchar[{\par\vskip 6pt\noindent\@pf}{\it\proofname:~ }}%
{{\unskip\nobreak~\hfill\it\qedsymbol}\par\vskip1.2ex}
\makeatother
\theoremstyle{plain}
  \newtheorem{thm}{Theorem}[section]
  \newtheorem{lem}[thm]{Lemma}
  \newtheorem{prop}[thm]{Proposition}
  \newtheorem{cor}[thm]{Corollary}
  \newtheorem{conj}[thm]{Conjecture}
\theoremstyle{definition}
  \newtheorem{defn}[thm]{Defninition}
  
  \newtheorem{expl}[thm]{Example}
  
\theoremstyle{remark}
  \newtheorem{rem}[thm]{Remark}

\numberwithin{equation}{section}
\def\homeo{\approx}
\def\cardinal{\mathbb{N}_{0}}
\def\ordinal{\mathbb{N}_{1}}

\def\real{\mathbb{R}}

\def\I{\mathbb{I}}
\def\TD{\mathbb{TD}}
\def\proj{\mathrm{pr}}

\def\emptyarg{}
\def\ad#1{\def\thisarg{#1}\mathrm{ad}\ifx\thisarg\emptyarg\else(#1)\fi}
\def\const#1{\def\thisarg{#1}\operatorname{\iota}\ifx\thisarg\emptyarg\else(\hskip1pt#1)\fi}
\def\open#1{\def\thisarg{#1}\ifx\thisarg\emptyarg{\Site{Open}}\else{\Site{C^{#1}\text-Open}}\fi}
\def\domain#1{\def\thisarg{#1}\ifx\thisarg\emptyarg{\Site{Domain}}\else{\Site{C^{#1}\text-Domain}}\fi}
\def\convex#1{\def\thisarg{#1}\ifx\thisarg\emptyarg{\Site{Convex}}\else{\Site{C^{#1}\text-Convex}}\fi}
\def\locconv#1{\def\thisarg{#1}\ifx\thisarg\emptyarg{\Site{LocConv}}\else{\Site{C^{#1}\text-LocConv}}\fi}
\def\polyhedron#1{\def\thisarg{#1}\ifx\thisarg\emptyarg{\Site{Polyhedron}}\else{\Site{C^{#1}\text-Polyhedron}}\fi}
\def\smoothology#1{\def\thisarg{#1}\ifx\thisarg\emptyarg{\Category{Diffeology}}\else{\Category{C^{#1}\text-Diffeology}}\fi}
\def\der#1by#2{\frac{\operatorname{\mathit{d}}\hspace{-0.1mm}#1}{\operatorname{\mathit{d}}\hspace{-0.1mm}#2}}
\def\nder#1by#2times#3{\frac{\operatorname{\mathit{d}}^{#3}\hspace{-0.2mm}#1}{\operatorname{\mathit{d}}\hspace{-0.4mm}{#2\,}^{#3}}}
\def\pder#1by#2{\frac{\operatorname{\partial}\hspace{-0.1mm}#1}{\operatorname{\partial}\hspace{-0.1mm}#2}}
\def\npder#1by#2times#3{\frac{\operatorname{{\partial\,}^{#3}}\hspace{-0.2mm}#1}{\operatorname{\partial}\hspace{-0.1mm}{#2}^{#3}}}
\def\diff#1{\ifx#1(\operatorname{\mathit{d}}\hspace{0.25mm}#1\else\operatorname{\mathit{d}}\hspace{-0.5mm}#1\fi}
\def\pdiff#1{\ifx#1(\operatorname{\partial}\hspace{0.25mm}#1\else\operatorname{\partial}\hspace{-0.25mm}#1\fi}

\def\Img#1{\operatorname{Im}\hskip1.5pt(\hskip1pt#1)}
\def\midvert{\, \mathstrut \vrule \,}
\makeatletter
\def\@supp[#1]{\operatorname{supp}\hskip1pt(\hskip.5pt#1)}
\def\supp{\@ifnextchar[{\@supp}{\operatorname{supp}\hskip.5pt}}
\makeatother
\def\Category#1{{\sf{#1}}}
\def\Site#1{{\sf{#1}}}
\def\Object#1{\operatorname{Obj}\,(\text{\small\(#1\)})}
\def\Morphism#1{\operatorname{Mor}_{\text{\small\,\(#1\)\,}}}
\def\Covering#1{\operatorname{Cov}_{\text{\small\(#1\)\,}}}

\def\manifold{\Category{Manifold}}
\def\diffeology{\Category{Diffeology}}

\def\topology{\Category{Topology}}
\def\gentop{\Category{NumGenTop}}

\def\sets{\Site{Set}}
\def\homeo{\approx}

\def\Map{\operatorname{Map}}

\def\Int{\operatorname{Int}}
\def\Cl{\operatorname{Cl}}

\def\ast{\hbox{\footnotesize$*$}}

\def\wdim{\operatorname{w\text-dim}}

\def\dom#1{\operatorname{dom}\hskip1.5pt(\hskip1pt#1)}

\def\colim{\operatorname{colim}\,}

\newcommand{\comp}{\smash{\lower-.1ex\hbox{\scriptsize$\circ$\,}}}
\def\differentiabletxt{smooth}

\def\smoothtxt{smooth}

\def\hooklongrightarrow{\lhook\joinrel\longrightarrow}

\begin{document}
\ifdefined\expand
\baselineskip21pt
\else
\ifdefined\narrow
\baselineskip15pt
\else
\baselineskip18pt
\fi
\fi
%
%
\title[Fat CW Complex]{A closed manifold is a fat CW complex}
%
%
\author[Iwase]{Norio IWASE}
\email[Iwase]{iwase@math.kyushu-u.ac.jp}
\address[Iwase]{Faculty of Mathematics, Kyushu University, Fukuoka 819-0395, Japan}
\author[Kojima]{Yuki Kojima}
\email[Kojima]{yukikojima8128@gmail.com}
\address[Kojima]
{Toyota Systems Corporation, JP Tower Nagoya 32F, 1-1-1 Meieki, Nakamura-ku, Nagoya-shi 450-6332, Japan}
%
%
%
%
%
\begin{abstract}
The main purpose of this paper is to introduce a new smooth version of a CW complex named a fat CW complex, and to show that it includes all closed manifolds, because existing smooth versions of CW complexes (e.g. \cite{MR4414309}) do not have such property. We also verify that de Rham theorem holds for a fat CW complex and that a regular CW complex is reflexive in the sense of Y.~Karshon, J.~Watts and P.~I-Zemmour. Further, any topological CW complex is topologically homotopy equivalent to a fat CW complex. So, a fat CW complex enjoys many nice properties.
\end{abstract}
%
%
\maketitle

\section{Introduction}\label{sect:Introduction}

When we try to import an idea from manifold theory to diffeology, we often encounter a problem caused by the fact that there are several different notions in diffeology corresponding to the idea in manifold theory or usual topology such as differentiable structures on simplices and cubes (see \cite{AX13115668,MR3913971,MR4712607}).
We believe that the categorical techniques help us to sort out such problems.

In this paper, $\diffeology$ stands for the category of diffeological spaces and smooth maps (see \cite{MR3025051} for example), which is cartesian-closed, complete and cocomplete.
We denote by $\manifold$ the category of smooth manifolds with or without boundary, which forms a full subcategory of $\smoothology{}$, where a manifold is assumed to be paracompact.

Our goal in this paper is to present the following picture in \diffeology{}:
\begin{center}
\begin{picture}(376,155)(5,-5)%
\put(7,-5)	{\begin{minipage}[b]{130mm}\begin{itembox}[l]{$\diffeology$}\vphantom{$\begin{array}{l}\Bigg(\\[3ex]\Bigg|\\[3ex]\Bigg)\end{array}$}\end{itembox}\end{minipage}}
\put(33,20)	{\begin{minipage}[b]{68.4mm}\begin{itembox}[l]{\text{\footnotesize Reflexive}}\vphantom{$\begin{array}{l}\Bigg(\\[-2ex]\Big)\end{array}$}\end{itembox}\end{minipage}}
\put(95,05)	{\begin{minipage}[b]{90mm}\begin{itembox}[r]{\text{\footnotesize de Rham}}\vphantom{$\begin{array}{l}\Bigg(\\[6ex]\Bigg)\end{array}$}\end{itembox}\end{minipage}}
\put(106,12)	{\begin{minipage}[b]{82mm}\begin{itembox}[l]{\text{\footnotesize Fat CW}}\vphantom{$\begin{array}{l}\Bigg(\\[-1ex]\Bigg)\end{array}$}\end{itembox}\end{minipage}}
\put(116,27)	{\begin{minipage}[b]{39.4mm}\begin{itembox}[c]{\footnotesize Regular CW}\vphantom{$\begin{array}{l}\Big(\\[-3ex]\Big)\end{array}$}\end{itembox}\end{minipage}}
\put(125,34)	{\begin{minipage}[b]{36mm}\begin{itembox}[l]{\footnotesize Closed Manifold\!}\vphantom{$\begin{array}{l}\Big(\\[-7ex]\Big)\end{array}$}\end{itembox}\end{minipage}}
\put(215.5,27)	{\begin{minipage}[b]{40mm}\begin{itembox}[r]{\text{\footnotesize Thin CW}}\vphantom{$\begin{array}{l}\Big(\\[-3ex]\Big)\end{array}$}\end{itembox}\end{minipage}}
\put(233,50)	{\footnotesize$\I$}
\put(216.5,46.5)	{\footnotesize(\hskip-.1em\ast\hskip-.1em)}
\put(245,46.5)	{\dbox{\footnotesize Topological CW}}
\end{picture}
\end{center}
We also expect that ``regular CW'' part includes all compact manifolds.

In the above picture, ``de Rham'' denotes the class of diffeological spaces where de Rham theorem holds, ``$\I$'' denotes the exotic interval given in \cite{MR4712607}, ``Topological CW'' denotes the class of diffeological spaces with topological homotopy types of topological CW complexes, and $(\ast)$ denotes the class of manifolds of dimension $0$.
For other notions in the picture except for a manifold, we give some explanations later in this paper.

\section{Foundations}

In this section, we present a brief introduction of a diffeological space by Souriau \cite{MR753860}, giving a slightly more categorical formulation than the original one in \cite{MR3025051} to show that the notion is so natural and also simple.
For the meaning of such categorical formulations, please have a look at \cite{nlab-site} and Baez-Hoffnung \cite{MR2817410}.
Later in this section, we also introduce a weaker version of ``dimension'' for our fat CW complexes. 

Let $\domain{}$ be the category of all open sets in $\real^{n}$ for all $n \ge 0$ and {\smoothtxt} functions between them, equipped with the set of all open coverings denoted by $\Covering{\domain{}}(U)$ on each object $U$ in $\domain{}$, i.e, $U$ is an open set in $\real^{n}$ for some $n \ge 0$.

For a set $X$, we have two contravariant functors $\mathcal{M}_{X}, \,\mathcal{K}_{X} : \domain{}^{\text{op}} \to \sets$ given by
\begin{enumerate*}
\vitem
$\mathcal{M}_{X}(U) = \Map(U,X)$ and $\mathcal{M}_{X}(\phi)(P)=P{\comp}\phi$, and
\vitem
$\mathcal{K}_{X}(U) = \{\,Q \in \mathcal{M}_{X}(U) \midvert \text{$Q$ is locally constant}\,\}$ and $\mathcal{K}_{X}(\phi)(Q)=Q{\comp}\phi$
\end{enumerate*}%
for any $U \in \Object{\Site{C}}$, $P \in \mathcal{M}_{X}(U)$, $Q \in \mathcal{K}_{X}(U)$ and $\phi \in \Morphism{\Site{C}}(V,U)$, 
where $Q : U \to X$ is said to be locally constant, if there exists a covering family $\{\,\psi_{\alpha} : V_{\alpha} \to U\,\}_{\alpha \in \Lambda}$ of $U$ such that $Q{\comp}\psi_{\alpha}$ is constant for all $\alpha \in \Lambda$.
Then we have $\mathcal{K}_{X} \subset \mathcal{M}_{X}$ as contravariant functors, i.e, $\mathcal{K}_{X}(U) \subset \mathcal{M}_{X}(U)$ and $\mathcal{K}_{X}(\phi) = \mathcal{M}_{X}(\phi)\vert_{\mathcal{K}_{X}(U)} : \mathcal{K}_{X}(U) \to \mathcal{K}_{X}(V) \subset \mathcal{M}_{X}(V)$.
We often call an element of $\mathcal{M}_{X}(U)$ a \textit{parametrization} of $X$ on $U \in \Object{\domain{}}$.

For a set $X$, we call $(X,\mathcal{D}_{X})$ a \textit{diffeological space}, if it satisfies the following conditions.
\begin{enumerate*}
\vitem[D1 (Smooth Compatibility)]\label{defn:smoothology'-1}
$\mathcal{D}_{X} : \domain{}^{\text{op}} \to \sets$ is a contravariant functor.
\vitem[D2 (Covering)]\label{defn:smoothology'-2}
$\mathcal{K}_{X}$ $\subset$ $\mathcal{D}_{X}$ $\subset$ $\mathcal{M}_{X}$ as contravariant functors.
\vitem[D3 (Locality)]\label{defn:smoothology'-3}
For given $U \in \Object{\domain{}}$ and $P \in \mathcal{M}_{X}(U)$, 
$P \in \mathcal{D}_{X}(U)$ if there exists $\{U_{\alpha}\}_{\alpha\in\Lambda} \in \Covering{\domain{}}(U)$ such that $P\vert_{U_{\alpha}} \in \mathcal{D}_{X}(U_{\alpha})$ for all $\alpha \in \Lambda$.
\end{enumerate*}%
We often call an element of $\mathcal{D}_{X}(U)$ a plot of $X$ on $U \in \Object{\domain{}}$, and \vspace{.5ex}$\mathcal{D} = \bigcup_{U}\mathcal{D}_{X}(U)$ is called a \textit{diffeology} on $X$ following \cite{MR753860,MR3025051}.

A map $f : X \to Y$ induces a natural transformation $f_{\!\ast} : \mathcal{M}_{X} \to \mathcal{M}_{Y}$ given by $f_{\!\ast}(P)=f{\comp}P$.
We say a map $f : X \to Y$ {\differentiabletxt}, if 
the natural transformation $f_{\!\ast} : \mathcal{M}_{X} \to \mathcal{M}_{Y}$ satisfies $f_{\!\ast}(\mathcal{D}_{X}(U)) \subset \mathcal{D}_{Y}(U)$ for every $U \!\in\! \Object{\domain{}}$.
We say $(X,A)$ a pair of diffeological spaces, if $X$ is a diffeological space and $A$ is its diffeological subspace.

For other basics including induction and subduction, please have a look at \cite{MR3025051}.

Now we introduce a \textit{weak dimension} for a pair of diffeological spaces $(X,A)$ with an induction $\iota : A \hookrightarrow X$:
a family $\mathcal{F}$ of smooth maps to $X$ is called a generalised generating family (or GGF) on $(X,A)$, if the following two conditions are satisfied.
\begin{enumerate*}
\vitem 
the domain $\dom{f}$ of $f \in \mathcal{F}$ is open in $\real_+^{d}$, $d \ge 0$, 
where $\real_+=[0,\infty)$.
\vitem 
$F : A \amalg \coprod_{f \in \mathcal{F}}\dom{f} \rightarrow X$ given by $F|_{A}=\iota$ and $F|_{\dom{f}}=f$, is a subduction.
\end{enumerate*}%
Let $\wdim\mathcal{F}$ be the smallest $d \ge 0$ such that $\dim\dom{f} \le d$ for any $f \in \mathcal{F}$.
\begin{defn}
$\wdim{(X,A)}=\min\,\{\,\wdim\mathcal{F} \mid \text{$\mathcal{F}$ is a GGF on $(X,A)$}\,\}$. 
\end{defn}
When $A=\emptyset$, we denote $\wdim{X}=\wdim{(X,\emptyset)}$.
If $M$ is a $d$-manifold with corners, we have $\wdim{M}=d$, while $\dim{M}=\infty$ (see \cite{MR3025051}).

\section{Smooth Handles}

Taking $D$-topology (see \cite{MR3025051}) gives a left-adjoint forgetful functor $T : \smoothology{} \to \topology{}$.
Let us denote by $\gentop$ the category of numerically generated topological spaces, introduced by Shimakawa-Yoshida-Haraguchi in \cite{MR3884529} as the image of $T$.
For a smooth manifold with or without boundary, $T(X)$ is often denoted again by $X$.
Let $\lambda : \real \to \real$ be a smooth function given as follows:
\par\vskip1ex\noindent\hfil$\displaystyle
\lambda(t) = \frac1{\alpha}{\cdot}\!\int_{0}^{t}\!\ell(3x){\cdot}\ell(3{-}3x)\diff{x},\quad
\alpha=\int_{0}^{1}\!\ell(3x){\cdot}\ell(3{-}3x)\diff{x},\quad
\ell(t) = \begin{cases}\,0,&t \!\le\! 0,\\\,e^{-\fracinlines1/t},&t\!>\!0.\end{cases}
$\hfil\par\vskip1ex\noindent
According to \cite{wolfram2}, $e^{\fracinlines4/3}{\cdot}\alpha = 0.55731493\cdots > \fracinline{11}/{20}$, and hence we have $\frac{\ell(\fracinlines3/2)^{2}}{\alpha} < \fracinline{20}/{11}$.
Further, we obtain $\ell'(t)=\frac1{t^{2}}{\cdot}\ell(t)$ if $t>0$, and $\lambda$ enjoys the following four properties.
\begin{enumerate*}%
\vitem\label{prty:lambda1} $\lambda(t)=0$ if $t \le 0$,
\hitem\label{prty:lambda3} $\lambda(t)+\lambda(1{-}t)=1$,
\hitem\label{prty:lambda2} $\lambda'(t)=\frac{1}{\alpha}{\cdot}\ell(3t){\cdot}\ell(3{-}3t)$, 
\vitem\label{prty:lambda4} $\der{}by{t}\ell(3t)=\frac1{3t^{2}}{\cdot}\ell(3t)$ if $t > 0$, and $\lambda''(t)=\frac{1-2t}{3t^{2}(1-t)^{2}}{\cdot}\lambda'(t)$ if $0 < t < 1$.
\end{enumerate*}%
By \ref{prty:lambda3}), we have $\lambda(\fracinline1/2)=\fracinline1/2$.
By \ref{prty:lambda4}), together with the fact that $\lambda(0)=0$ and $\lambda(\fracinline1/2)=\fracinline1/2$, we obtain that $0 < \lambda(t) < t$ if $0 < t < \fracinline1/2$.
It is also follows from \ref{prty:lambda1}) and \ref{prty:lambda3}) that $\lambda(t)=1$ if $t \ge 1$.
We obtain the following proposition by \ref{prty:lambda2}) and \ref{prty:lambda4}).
\begin{prop}\label{prop:maximumlambda'}
$\lambda'(t) \le \lambda'(\fracinline1/2)=\frac{\ell(\fracinlines3/2)^{2}}{\alpha} < \fracinline{20}/{11}$ for all $t \in \real$.
\end{prop}
Let $\phi : \real \to [0,\infty)$ be a smooth function given as follows. 
\par\vskip1ex\noindent\hfil$\displaystyle
\phi(t) = \int_{0}^{\fracinlines1/2+t}\!\!\lambda(x)\diff{x}
$\hfil\par\vskip1ex\noindent
Then by \ref{prty:lambda1}) through \ref{prty:lambda4}), $\phi$ enjoys the following another four properties.

\begin{enumerate*}%
\addtocounter{enumi}{4}
\vitem\label{prty:phi1}
$\phi(t)=0$ if $t \le -\fracinline1/2$,
\hitem\label{prty:phi3}
$\phi(t)=t$ if $t \ge \fracinline1/2$,
\hitem\label{prty:phi2}
$0<\phi(0)<\fracinline1/8$,
\vitem\label{prty:phi5}
$\phi$ is monotone increasing on $\real$ and strictly monotone on $[-\fracinline1/2,\infty)$.
\end{enumerate*}%
Then we obtain the following proposition.
\begin{prop}\label{prop:relation-phi}
$\phi(t)-\phi(-t)=t$ for all $t \in \real$ and $\phi(t)+\phi(-t)=|t|$ if $|t| \ge \fracinline1/2$.
\end{prop}\vspace{-2ex}
\begin{proof}
Firstly, $\phi(0)-\phi(0)=0$, and $\der{}by{t}(\phi(t)-\phi(-t))=\lambda(\fracinline1/2{+}t)+\lambda(\fracinline1/2{-}t)=1$ by \ref{prty:lambda3}).
Thus $\phi(t)-\phi(-t)=t$.
Secondly, if $t \ge \fracinline1/2$, then $\phi(t)=t$ by \ref{prty:phi3}) and $\phi(-t)=0$ by \ref{prty:phi1}), and hence we have $\phi(t)+\phi(-t) = t = |t|$.
If $t \le \fracinline{-1}/2$, then $\phi(t)=0$ by \ref{prty:phi1}) and $\phi(-t)=-t=|t|$ by \ref{prty:phi3}) as well, and hence we have $\phi(t)+\phi(-t) = |t|$ if $|t| \ge \fracinline1/2$.
\end{proof}

Now we are ready to introduce the following manifolds and their boundaries.
\begin{align*}&
\mathbb{S}^{n-1} = \{\, \mathbb{u} \in \real^{n} \mid \Vert{\mathbb{u}}\Vert \ge 1 \,\} = \real^{n} \smallsetminus \Int{D^{n}},\quad n \ge 1,
\\&
\partial\mathbb{S}^{n-1} = \{\,\mathbb{u} \in \real^{n} \mid \Vert{\mathbb{u}}\Vert = 1\,\} = S^{n-1},
\\&
\mathbb{D}^{n,m}=\{\,(\mathbb{u},\mathbb{v}) \in \real^{n} \times \real^{m} \mid \Vert{\mathbb{u}}\Vert \ge 1-\phi(2{-}\Vert{\mathbb{u}}\Vert{-}\Vert{\mathbb{v}}\Vert)\,\},
\\&
\partial\mathbb{D}^{n,m}=\{\,(\mathbb{u},\mathbb{v}) \in \real^{n} \times \real^{m} \mid \Vert{\mathbb{u}}\Vert = 1-\phi(2{-}\Vert{\mathbb{u}}\Vert{-}\Vert{\mathbb{v}}\Vert)\,\},
\\&
\mathbb{S}^{n-1,m} = \{\,(\mathbb{u},\mathbb{v}) \in \real^{n} \times \real^{m} \mid \Vert{\mathbb{u}}\Vert \ge 1\,\} = \mathbb{S}^{n-1} \times \real^{m},
\\&
\partial\mathbb{S}^{n-1,m} = \{\,(\mathbb{u},\mathbb{v}) \in \real^{n} \times \real^{m} \mid \Vert{\mathbb{u}}\Vert = 1\,\} = S^{n-1} \times \real^{m},
\end{align*}
where $D^{n} = \{\, \mathbb{v} \in \real^{n} \mid \Vert{\mathbb{v}}\Vert \le 1 \,\} \supset S^{n-1} = \{\, \mathbb{v} \in \real^{n} \mid \Vert{\mathbb{v}}\Vert = 1 \,\}$, and $\mathbb{S}^{-1} = S^{-1} = \emptyset$.
Further, let $\partial_{0}\mathbb{S}^{n-1,m} = \{\, (\mathbb{u},\mathbb{v}) \in \mathbb{S}^{n-1,m} \mid \Vert{\mathbb{u}}\Vert = 1, \ \Vert{\mathbb{v}}\Vert \le \fracinline3/2 \,\} \subset \partial\mathbb{S}^{n-1,m}
$ and $\mathbb{S}_{0}^{n-1,m} = \mathbb{S}^{n-1,m} \smallsetminus \partial_{0}\mathbb{S}^{n-1,m}$.
The following is an image of $(\mathbb{D}^{n,m},\mathbb{S}^{n-1,m},\partial\mathbb{S}^{n-1,m},\partial_{0}\mathbb{S}^{n-1,m})$.

\begin{picture}(400,120)(-105,0)\thicklines
\put(0,0){\includegraphics[width=70mm]{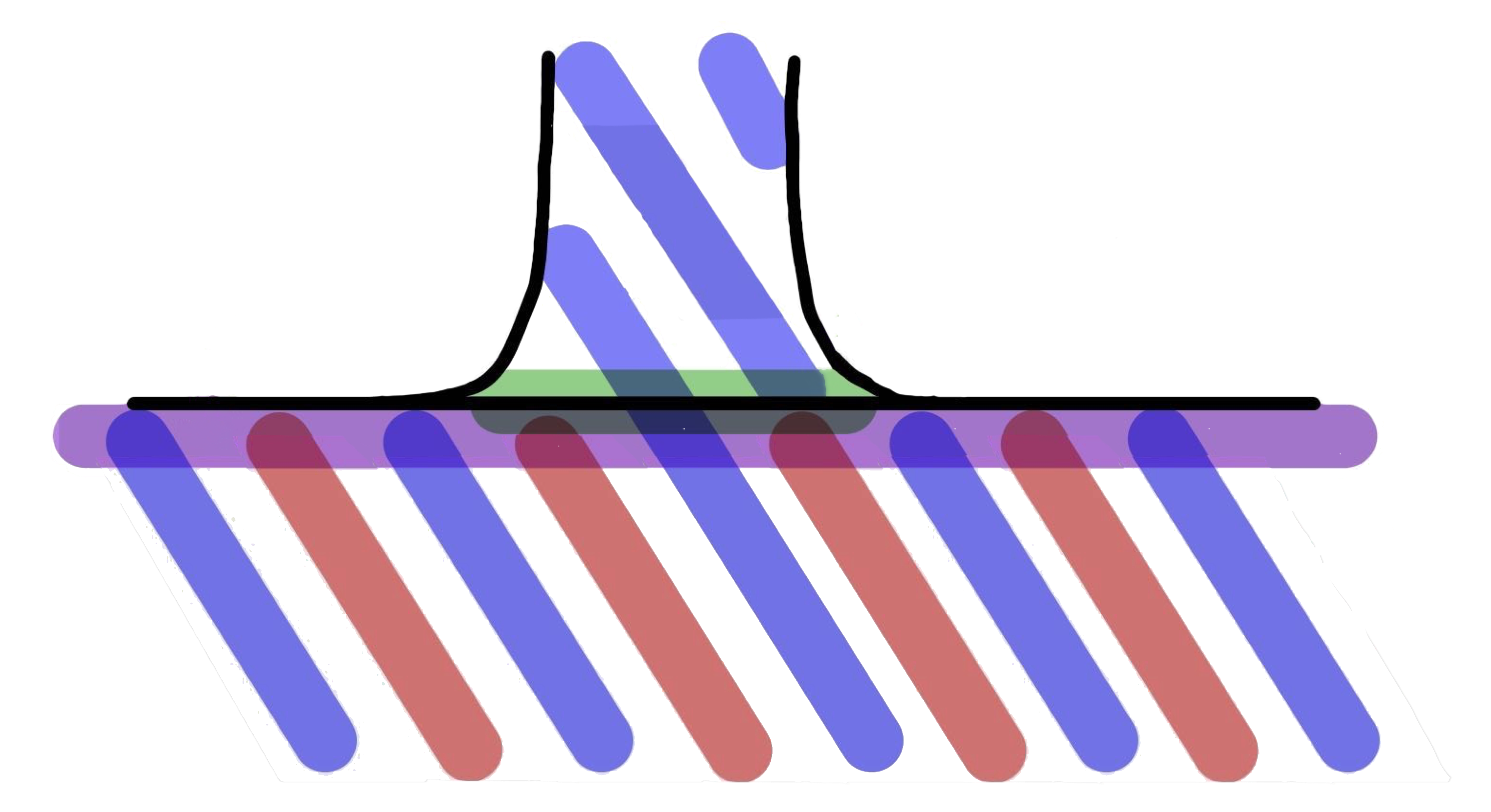}}
\put(-40,53)	{\makebox(0,0){\color{slateblue}$\mathbb{D}^{n,m}\,\,\left\{\vphantom{\begin{array}{c}\Big(\\\Big(\\\Big)\\\Big)\end{array}}\right.$}}
\put(200,52)	{\makebox(0,0)[l]{\color{violet}{\large$\longleftarrow$}\,\, $\partial\mathbb{S}^{n-1,m}$}}
\put(139,82)	{\makebox(0,0){\color{forestgreen}$\partial_{0}\mathbb{S}^{n-1,m}$}}
\put(116,77)	{\color{forestgreen}\vector(-1,-1){19}}
\put(220,25)	{\makebox(0,0)[l]{\color{darkbrown}$\left.\vphantom{\begin{array}{c}\Big(\\[.5ex]\big)\end{array}}\right\}\,\, \mathbb{S}^{n-1,m}$}}
\end{picture}

\begin{rem}\label{rem:sphere}
Then we see that $\mathbb{S}^{n-1} = \real^{n} \smallsetminus \Int{D^{n}} \homeo S^{n-1} \times [1,\infty)$, where $\mathbb{u} \in \mathbb{S}^{n-1}$ corresponds to $(\frac1{\Vert{\mathbb{u}}\Vert}\mathbb{u},\Vert{\mathbb{u}}\Vert) \in S^{n-1} \times [1,\infty)$.
Hence, $\Int\mathbb{S}^{n-1} = \real^{n} \smallsetminus D^{n} \homeo S^{n-1} \times (0,\infty)$.
\end{rem}

\begin{rem}\label{rem:m=0}
Since $\phi(u{-}2)-\phi(2{-}u)=u-2$ by Proposition \ref{prop:relation-phi}, $1-\phi(2{-}u)=u-1-\phi(u{-}2) \le u-1<u$.
Thus $\mathbb{D}^{n,0} = \Int\mathbb{D}^{n,0} = \real^{n}$, $\mathbb{S}^{n-1,0} = \mathbb{S}^{n-1}$ and $\Int\mathbb{S}^{n-1,0} = \Int\mathbb{S}^{n-1}$.
\end{rem}

\begin{prop}\label{prop:boundary-fathandle}
Let $(\mathbb{u},\mathbb{v}) \in \real^{n} \times \real^{m}$. 
Then $(\mathbb{u},\mathbb{v}) \in \partial\mathbb{D}^{n,m}$ if and only if $(\Vert{\mathbb{u}}\Vert,\Vert{\mathbb{v}}\Vert)=(1-\phi(-t),1+\phi(t))$ for $t=\Vert{\mathbb{u}}\Vert+\Vert{\mathbb{v}}\Vert-2$ .
\end{prop}
\begin{proof}
Assume $(\mathbb{u},\mathbb{v}) \in \partial\mathbb{D}^{n,m}$.
Then $\Vert{\mathbb{u}}\Vert = 1-\phi(2{-}\Vert{\mathbb{u}}\Vert{-}\Vert{\mathbb{v}}\Vert)$, and by Proposition \ref{prop:relation-phi}, $\Vert{\mathbb{v}}\Vert = 1+\phi(\Vert{\mathbb{u}}\Vert{+}\Vert{\mathbb{v}}\Vert{-}2)$, and hence we obtain $(\Vert{\mathbb{u}}\Vert,\Vert{\mathbb{v}}\Vert)=(1-\phi(-t),1+\phi(t))$ for $t=\Vert{\mathbb{u}}\Vert+\Vert{\mathbb{v}}\Vert-2$.
The converse is clear.
\end{proof}

\begin{prop}
$\mathbb{S}^{n-1} \times \real^{m} \subset \mathbb{D}^{n,m} \supset \real^{n} \times D^{m}$.
\end{prop}
\begin{proof}
Let $(\mathbb{u},\mathbb{v}) \in \mathbb{S}^{n-1} \times \real^{m}$.
Then $\Vert{\mathbb{u}}\Vert \ge 1 \ge 1-\phi(2{-}\Vert{\mathbb{u}}\Vert{-}\Vert{\mathbb{v}}\Vert)$ by the definition of $\phi$, and hence $(\mathbb{u},\mathbb{v}) \in \mathbb{D}^{n,m}$.
Let $(\mathbb{u},\mathbb{v}) \in \real^{n} \times D^{m}$.
Then we have $\Vert{\mathbb{v}}\Vert \le 1$ and $\Vert{\mathbb{u}}\Vert - 1 \ge \Vert{\mathbb{u}}\Vert + \Vert{\mathbb{v}}\Vert - 2 = \phi(\Vert{\mathbb{u}}\Vert{+}\Vert{\mathbb{v}}\Vert{-}2) - \phi(2{-}\Vert{\mathbb{u}}\Vert{-}\Vert{\mathbb{v}}\Vert) \ge - \phi(2{-}\Vert{\mathbb{u}}\Vert{-}\Vert{\mathbb{v}}\Vert)$ by Proposition \ref{prop:relation-phi} and the definition of $\phi$, and hence $(\mathbb{u},\mathbb{v}) \in \mathbb{D}^{n,m}$.
\end{proof}

We obtain the following theorems whose proofs shall be given in Appendix.

\begin{thm}\label{thm:diffeotypeofhandle}
There is a diffeomorphism $\Psi_{n,m} : \real^{n} \times D^{m} \to \mathbb{D}^{n,m}$.
\end{thm}

One can easily see that $\mathbb{D}^{n,m} \smallsetminus \partial_{0}\mathbb{S}^{n-1,m}$ is a union of two disjoint subsets $\mathbb{D}^{n,m} \smallsetminus \mathbb{S}^{n-1,m}$ and $\mathbb{S}_{0}^{n-1,m}$ both of which are open in $\mathbb{D}^{n,m}$.
In fact, we can easily see that $\partial_{0}\mathbb{S}^{n-1,m} = \widehat\Phi_{n,m}(S^{n-1} \times D^{m})$, $\mathbb{D}^{n,m} \smallsetminus \mathbb{S}^{n-1,m} = \widehat\Phi_{n,m}(\Int{D^{n}} \times D^{m})$ and $\mathbb{S}_{0}^{n-1,m} = \widehat\Phi_{n,m}((\real^{n} \smallsetminus D^{n}) \times D^{m})$.

\begin{thm}\label{thm:smoothtypeofhandle}
There is a smooth homeomorphism $\widehat\Phi_{n,m} : (\real^{n} \times D^{m},\mathbb{S}^{n-1} \times D^{m},S^{n-1} \times D^{m}) \to (\mathbb{D}^{n,m},\mathbb{S}^{n-1,m},\partial_{0}\mathbb{S}^{n-1,m})$ which is diffeomorphic apart from $S^{n-1} \times D^{m}$.
This implies that $(\real^{n} \times D^{m},\mathbb{S}^{n-1} \times D^{m})$ has the topological homotopy type of $(D^{n},S^{n-1})$.
\end{thm}

\section{Fat CW Complex}

Let $(X,A)$ be a pair of diffeological spaces, and $\{(X^{(n)},A)\}$ a series of diffeological spaces, $n \ge -1$.

\begin{defn}\label{defn:smoothCW}
The pair $(X,A)$ is said to be a \textit{relative fat CW complex} with \textit{fat $n$-skeleta} $(X^{(n)},A)$, if there exists a series of smooth attaching maps 
\par\vskip1ex\noindent\hfil$\displaystyle
h_{n} : \underset{j \in J_{n}}{\textstyle\coprod} \mathbb{S}^{n-1,m_{n}(j)} \to X^{(n-1)},
$\hfil\par\vskip1ex\noindent
$n \ge 0$, where $J_{n}$ is an index set and $m_{n}$ is a map from $J_{n}$ to $\cardinal$ the set of non-negative integers, satisfying the following conditions:
\begin{enumerate}\setcounter{enumi}{-1}
\item 
$X^{(-1)}=A$.
\item 
For $n \ge 0$, $X^{(n)}$ is the pushout of the inclusion $i_{n} : \underset{j \in J_{n}}\coprod \mathbb{S}^{n-1,m_{n}(j)} \hookrightarrow \underset{j \in J_{n}}\coprod \mathbb{D}^{n,m_{n}(j)}$ and the attaching map $h_{n} : \underset{j \in J_{n}}{\textstyle\coprod} \mathbb{S}^{n-1,m_{n}(j)} \rightarrow X^{(n-1)}$ in $\smoothology{}$:
\par\vskip1ex\noindent\hfil$\begin{tikzcd}
{\underset{j \in J_{n}}{\textstyle\coprod} \mathbb{S}^{n-1,m_{n}(j)}}
\arrow[r,"h_{n}"]
\arrow[d,hook,"i_{n}"']
&
{X^{(n-1)}}
\arrow[d,hook,"\hat{i}_{n}"]
\\[1ex]
{\underset{j \in J_{n}}\coprod \mathbb{D}^{n,m_{n}(j)}}
\arrow[r,"\hat{h}_{n}"]
&
{X^{(n)}}
\end{tikzcd}%
$\hfil\par\vskip1ex
\item
$X$ is a diffeological colimit of $X^{(n)}$, $n \ge 0$.
\end{enumerate}
\end{defn}
A relative fat CW complex $(X,A)$ with $m_{k}=0$ for all $k \ge 0$ is called a \textit{relative thin CW complex}.
If $(X,\emptyset)$ is a relative fat CW complex, then $X$ is called a \textit{fat CW complex}, and if $(X,\emptyset)$ is a relative thin CW complex, then $X$ is called a \textit{thin CW complex}.

\begin{prop}\label{prop:dimension}
Let $M_{k}=\max\,\{\,m_{k}(j) \mid j \in J_{k}\,\}$, $k \ge 0$.
Then we have $\wdim{X^{(n)}}$ $\le$ $\max\,\{\,M_{k}\!+\!k \mid 0 \!\le\! k \!\le\! n\,\}$.
Hence $\wdim{(X^{(n)},A)}$ $\le$ $n$, provided that $X$ is relative thin.%
\end{prop}

\begin{expl}\label{expl:main1}
Let $\pi_{set} : \real \to \real$ be the (continuous) map defined as follows:\vspace{-1ex}
$$
\pi_{set}(t)=\max\{\,0, \min\,\{\,t,1\,\}\,\}=\min\,\{\,1, \max\,\{\,t,0\,\}\,\} \in [0,1] \subset \real.
$$\vskip-1ex
Then the diffeological quotient $\I=\real/\pi_{set}$ is a thin CW complex with a subduction $\pi : \real \to \I$ by $\pi(t)=[\pi_{set}(t)]$, $t \in \real$.
Thus we also have $\dim\I=1$.
\end{expl}

\begin{expl}
Let $(\check{X},\{\check{X}_{n}\})$ be a smooth CW complex with smooth attaching maps $\check{h}_{n} : S^{n-1} \to \check{X}^{(n-1)}$, $n \ge 0$.
Then there is a fat CW complex $(X,\{X^{(n)}\})$ and smooth injections $j_{n} : \check{X}_{n} \hookrightarrow X^{(n)}$, $n \ge 0$ with smooth attaching maps $h_{n}=j_{n-1}{\comp}\check{h}_{n}{\comp}p_{n-1} : \mathbb{S}^{n-1,0}$ $=$ $\mathbb{S}^{n-1}$ $\cong$ $S^{n-1} \times [1,\infty)$ $\xrightarrow{\proj_{1}}$ $S^{n-1}$ $\xrightarrow{\check{h}_{n}}$ $\check{X}_{n}$ $\overset{j_{n-1}}\hooklongrightarrow$ $X^{(n-1)}$, $n \ge 0$, by Remarks \ref{rem:m=0} and \ref{rem:sphere}.
Since smooth injections $j_{n}$ are topological homeomorphisms, each skeleton $\check{X}_{n}$ is topologically a fat CW complex, $n \ge 0$.
Further, we have $\Int{X^{(n)}} = X^{(n)}$, $n \ge 0$.
\vspace{1ex}
\par\noindent\hfil\hfil$\begin{tikzcd}
\underset{j \in J_{n}}\coprod\mathbb{S}^{n-1,0} \arrow[rrr,"h_{n}"] \arrow[ddd, hook] &                                          &                               & X^{(n-1)} \arrow[ddd, hook] \\
                            & \underset{j \in J_{n}}\coprod S^{n-1} \arrow[r,"\check{h}_{n}"] \arrow[d, hook] \arrow[lu, hook, "j'_{n-1}"'] & \check{X}_{n-1} \arrow[d, hook] \arrow[ru, hook, "j_{n-1}"] &                \\
                            & \underset{j \in J_{n}}\coprod D^{n} \arrow[r] \arrow[ld, hook]            & \check{X}_{n} \arrow[rd, hook, "j_{n}"]            &                \\
\underset{j \in J_{n}}\coprod\mathbb{D}^{n,0} \arrow[rrr]              &                                          &                               & X^{(n)}             
\end{tikzcd}
$\hfil
\end{expl}

\section{Basic Properties}

In this section, we show some basic properties of a fat CW complex.

Let $X$ be a collection of pairs $X^{(n)}$ of diffeological spaces and attaching maps $h_{n} : \underset{j \in J_{n}}{\textstyle\coprod} \mathbb{S}^{n-1,m_{n}(j)} \to X^{(n-1)}$.
Then it is clear that $X^{(n-1)}$ is a closed subset of $X^{(n)}$ under the $D$-topology.
We first show that a fat CW complex is a paracompactum.

\begin{prop}\label{prop:paracompact}
The $D$-topology of $X$ is paracompact and Hausdorff.
\end{prop}
\begin{proof}
We show that $X^{(n)}$ is paracompact and Hausdorff by induction on $n \ge -1$.
When $n=-1$, there is nothing to do.
So, we assume that we have done up to $n\!-\!1$, $n \ge 0$.
Firstly, $\mathbb{D}^{n,m} \homeo \real^{n}\times D^{m}$ is paracompact and Hausdorff as a closed subset of $\real^{n}\times\real^{m}$, and so is its closed subset $\mathbb{S}^{n-1,m}$.
Thus the coproducts $\underset{j \in J_{n}}{\textstyle\coprod} \mathbb{D}^{n,m_{n}(j)}$ and $\underset{j \in J_{n}}{\textstyle\coprod} \mathbb{S}^{n-1,m_{n}(j)}$ are paracompact and Hausdorff.
Secondly, since $i_{n} : \underset{j \in J_{n}}{\textstyle\coprod} \mathbb{S}^{n-1,m_{n}(j)} \hookrightarrow \underset{j \in J_{n}}{\textstyle\coprod} \mathbb{D}^{n,m_{n}(j)}$ is a closed embedding in \topology{}, the pushout $X^{(n)}$ of $i_{n}$ and $\hat{i}_{n} : X^{(n-1)} \to X^{(n)}$ is also paracompact and Hausdorff.
Finally, since ``paracompact and Hausdorff''-ness is preserved under taking colimit of closed embeddings, $X = \colim X^{(n)}$ is paracompact and Hausdorff.
\end{proof}

We then consider the inclusion map $\hat{i}_{n} : X^{(n-1)} \hookrightarrow X^{(n)}$, $n \!\ge\! 0$.

\begin{prop}\label{prop:hat_{h_{n}}}
For each $n \!\ge\! 0$, the map $\hat{i}_{n} : X^{(n-1)} \hookrightarrow X^{(n)}$ is an induction.
\end{prop}
\begin{proof}
Assume that $P : U \to X^{(n)}$ is a plot with its image contained in $\Img{\hat{i}_{n}}$.
Since there is a subduction $X^{(n-1)} \amalg \underset{j \in J_{n}}\coprod \mathbb{D}^{n,m_{n}(j)} \twoheadrightarrow X^{(n)}$ by the definition of $X^{(n)}$, there is an open cover $\{U_{\alpha}\}$ of $U$ such that $P|_{U_{\alpha}}$ can be pulled back to a plot from $U_{\alpha}$ to either $X^{(n-1)}$ or $\mathbb{D}^{n,m_{n}(j)}$ for some $j \in J_{n}$.
\par
In the first case, $P|_{U_{\alpha}}$ is pulled back to  a plot $P_{\alpha} : U_{\alpha} \to X^{(n-1)}$ as $P|_{U_{\alpha}} = \hat{i}_{n}{\comp}P_{\alpha}$.
\par
In the second case, $P|_{U_{\alpha}}$ is pulled back to a plot $P_{\alpha} : U_{\alpha} \to \mathbb{D}^{n,m_{n}(j)}$ for some $j \in J_{n}$ as $P|_{U_{\alpha}} = \hat{h}_{n}{\comp}P_{\alpha}$.
Then $\Img{P_{\alpha}} \subset (\hat{h}_{n})^{-1}(\Img{\hat{i}_{n}}) = \mathbb{S}^{n-1,m_{n}(j)}$ in $\mathbb{D}^{n,m_{n}(j)}$.
Since $i_{n} : \mathbb{S}^{n-1,m_{n}(j)} \hookrightarrow \mathbb{D}^{n,m_{n}(j)}$ is an induction, we may assume that $P_{\alpha} : U_{\alpha} \to \mathbb{S}^{n-1,m_{n}(j)}$ is a plot in $\mathbb{S}^{n-1,m_{n}(j)}$, and hence we obtain $P|_{U_{\alpha}} = \hat{h}_{n}{\comp}P_{\alpha} = \hat{i}_{n}{\comp}h_{n}{\comp}P_{\alpha}$.
\par
In either case, $P|_{U_{\alpha}}$ is pulled back to a plot in $X^{(n-1)}$.
Thus $\hat{i}_{n}$ is an induction.
\end{proof}

\section{Smooth Functions}\label{sect:EMSF}

\begin{thm}\label{thm:main-partition}
For any open covering $\mathcal{U}$ on a fat CW complex, there is a smooth partition of unity subordinate to $\mathcal{U}$.
\end{thm}
\begin{proof}
By Proposition \ref{prop:paracompact}, we may assume that $\mathcal{U}$ is of locally finite.
Let $X$ be a fat CW complex with fat skeletons $\{X^{(n)}\}$, and $\mathcal{U}_{n}=\{\,U \cap X^{(n)} \mid U \in \mathcal{U}\,\}$.
We show that there is a smooth partition of unity on $X^{(n)}$ subordinate to $\mathcal{U}_{n}$ by induction on $n \ge 0$.
\par(Case $n\!=\!0$)
Since $D^{m}$ is a manifold with boundary, we have nothing to do.
\par(Case $n\!>\!0$)
By induction hypothesis, there is a smooth partition of unity $\{\,\rho_{U}\,\}_{U \in \mathcal{U}}$ on $X^{(n-1)}$ subordinate to $\mathcal{U}_{n-1}$.
By pulling back $\rho_{U}$'s by $h^{j}_{n} : \mathbb{S}^{n-1,m_{n}(j)} \to X^{(n-1)}$ the restriction of $h_{n} : \underset{j \in J_{n}}{\coprod} \mathbb{S}^{n-1,m_{n}(j)} \to X^{(n-1)}$ to $\mathbb{S}^{n-1,m_{n}(j)}$, we obtain a smooth partition of unity $\{\rho_{U}{\comp}h^{j}_{n}\}_{U \in \mathcal{U}}$ on $\mathbb{S}^{n-1,m_{n}(j)}$, $j \in J_{n}$ subordinate to $\mathcal{U}^{j}_{n} = \{\,h_{n}^{-1}(U) \cap \mathbb{S}^{n-1,m_{n}(j)} \mid U \in \mathcal{U} \,\}$ of $\mathbb{S}^{n-1,m_{n}(j)}$, such that we obtain $\supp{\rho_{U}{\comp}h^{j}_{n}} \subset h_{n}^{-1}(U)$.
Then we take a smooth extension $\hat{\rho}^{j}_{U} : \mathbb{D}^{n,m_{n}(j)} \to \real$ of $\rho_{U}{\comp}h^{j}_{n}$ such that $\supp{\hat{\rho}^{j}_{U}} \subset \hat{h}_{n}^{-1}(U)$.
We also have a neighbourhood $V_{\!j}$ of $\mathbb{S}^{n-1,m_{n}(j)}$ in $\mathbb{D}^{n,m_{n}(j)}$ on which $\underset{V \in \mathcal{U}}\sum \hat{\rho}^{j}_{V} \not= 0$.

Let $\hat{\mathcal{U}}^{j}_{n}=\{\,\hat{h}_{n}^{-1}(U) \cap \mathbb{D}^{n,m_{n}(j)} \mid U \in \mathcal{U}\,\}$ an open cover of $\mathbb{D}^{n,m_{n}(j)}$.
Since $\mathbb{D}^{n,m_{n}(j)}$ is a manifold with boundary, we have a smooth partition of unity $\{\hat\sigma_{U}\}_{U \in \mathcal{U}}$ on $\mathbb{D}^{n,m_{n}(j)}$ subordinate to $\hat{\mathcal{U}}^{j}_{n}$ so that $\supp{\hat\sigma_{U}} \subset \hat{h}_{n}^{-1}(U) \cap \mathbb{D}^{n,m_{n}(j)}$ and $\underset{V \in \mathcal{U}}\sum \hat\sigma_{V} \not= 0$.

Since $\mathcal{V}_{\!j}=\{\,V_{\!j},\mathbb{D}^{n,m_{n}(j)} \smallsetminus \mathbb{S}^{n-1,m_{n}(j)}\,\}$ is an open covering of $\mathbb{D}^{n,m_{n}(j)}$, and hence there is a smooth partition of unity $\{\,\chi_{1}, \chi_{2}\,\}$ subordinate to $\mathcal{V}_{\!j}$:
\par\vskip1ex\noindent\hfil$\displaystyle
\chi_{1}+\chi_{2}=1,\qquad \begin{cases}\,
\chi_{1} : \mathbb{D}^{n,m_{n}(j)} \to \real,\quad
\supp{\chi_{1}} \subset \mathbb{D}^{n,m_{n}(j)} \smallsetminus \mathbb{S}^{n-1,m_{n}(j)},
\\[.5ex]\,
\chi_{2} : \mathbb{D}^{n,m_{n}(j)} \to \real,\quad
\supp{\chi_{2}} \subset V_{\!j},
\end{cases}
$\hfil\par\vskip1ex\noindent
Using the above functions, we define smooth functions $\sigma^{j}_{U}$ on $\mathbb{D}^{n,m_{n}(j)}$ as follows.
\par\vskip1ex\noindent\hfil$\displaystyle
\check\sigma^{j}_{U}(\mathbb{x})=\begin{cases}\,
\chi_{1}(\mathbb{x}){\cdot}\hat\sigma^{j}_{U}(\mathbb{x}) + \chi_{2}(\mathbb{x}){\cdot}\hat{\rho}^{j}_{U}(\mathbb{x}),&\mathbb{x} \in V_{\!j},
\\[.5ex]\,
\hat\sigma^{j}_{U}(\mathbb{x}),&\mathbb{x} \in \mathbb{D}^{n,m_{n}(j)} \smallsetminus \supp{\chi_{2}}.
\end{cases}
$\hfil\par\vskip1ex\noindent
Hence $\supp{\check\sigma^{j}_{U}} \subset \supp{\hat\sigma_{U}} \cup \supp{\hat{\rho}^{j}_{U}} \subset \hat{h}_{n}^{-1}(U)$ and $\underset{V \in \mathcal{U}}\sum \check\sigma^{j}_{V} \not= 0$.
Then we define %
\par\vskip1ex\noindent\hfil$\displaystyle
\sigma^{j}_{U}(\mathbb{x}) = \frac{\check\sigma^{j}_{U}(\mathbb{x})}{\underset{V \in \mathcal{U}}\sum \check\sigma^{j}_{V}(\mathbb{x})},\quad \mathbb{x} \in \mathbb{D}^{n,m_{n}(j)},
$\hfil\par\vskip1ex\noindent
which gives a smooth partition of unity on $\mathbb{D}^{n,m_{n}(j)}$ subordinate to $\hat{\mathcal{U}}^{j}_{n}$ and is also an extension of a smooth partition of unity $\{\rho_{U}{\comp}h^{j}_{n}\}_{U \in \mathcal{U}}$ on $\mathbb{S}^{n-1,m_{n}(j)}$ subordinate to $\mathcal{U}^{j}_{n}$.
Smooth maps $\sigma^{j}_{U}$, $j \in J_{n}$ and $\rho_{U}$ are compatible data for us to obtain a smooth partition of unity on the pushout $X^{(n)}$.
\end{proof}

\begin{cor}
For a fat CW complex, the de Rham theorem holds.
\end{cor}

Let $X_{0}^{(n-1)}$ $=$ $X^{(n)} \smallsetminus \hat{h}_{n}(\underset{j \in J_{n}}\coprod(\mathbb{D}^{n,m_{n}(j)} \smallsetminus \mathbb{S}_{0}^{n-1,m_{n}(j)}))$.
Since $\mathbb{D}^{n,m_{n}(j)} \smallsetminus \mathbb{S}_{0}^{n-1,m_{n}(j)}$ is compact in $\mathbb{D}^{n,m_{n}(j)}$, so is $\hat{h}_{n}(\underset{j \in J_{n}}\coprod(\mathbb{D}^{n,m_{n}(j)} \smallsetminus \mathbb{S}_{0}^{n-1,m_{n}(j)}))$ in $X^{(n)}$, and hence $X_{0}^{(n-1)}$ is $D$-open in $X^{(n)}$.
Since $X_{0}^{(n-1)}$ $=$ $X^{(n-1)} \smallsetminus h_{n}(\underset{j \in J_{n}}\coprod\partial_{0}\mathbb{S}^{n-1,m_{n}(j)}) \subset X^{(n-1)}$, $X_{0}^{(n-1)}$ is $D$-open in $X^{(n-1)}$, too.

Following \cite{MR4414309}, we say that a diffeological space $X$ has enough many smooth functions, if the $D$-topology of $X$ has an open base of the form $\pi^{-1}(0,1)$, where $\pi$ is a smooth function on $X$.
For example, by J.~Watts \cite{WattsThesis} and P.~I-Zemmour \cite{MR3025051}, a smooth manifold has enough many smooth functions.

\begin{thm}\label{thm:main-emsf}
A fat CW complex has enough many smooth functions.
\end{thm}
\begin{proof}
Let $X$ be a fat CW complex and $U$ be an open neighbourhood of an element $\mathbb{a} \in X$.
Then, since $X$ is Hausdorff, $V = X \smallsetminus \{\mathbb{a}\}$ is open, and hence $\mathbb{U} = \{U,V\}$ is an open covering of $X$.
Hence, by Theorem \ref{thm:main-partition}, there is a smooth partition of unity $\{\rho_{U},\rho_{V}\}$ subordinate to $\mathbb{U}$.
Put $f=\rho_{U}$, and we have done.
\end{proof}

\begin{defn}
A parametrization $P : U \to X$ of a diffeological space $X$ is said to be a pre-plot if it satisfies the following condition.
\begin{itemize}
\item for any smooth function $f : X \to \real$, $f{\comp}P$ is smooth in the ordinary sense.
\end{itemize}
\end{defn}

\begin{prop}\label{prop:main-emsf}
Let $X$ be a fat CW complex, and $n \ge 0$.
If a parametrization $P : U \to X^{(n)}$ is a pre-plot, then $P$ is continuous w.r.t. $D$-topology.
\end{prop}
\begin{proof}
It is sufficient to show that, for any open subset $O \subset X$, $P^{-1}(O)$ is open in $U$.
Assume $\mathbb{a} \in P^{-1}(O)$, and hence $P(\mathbb{a}) \in O$.
Then by Theorem \ref{thm:main-emsf}, there exists a smooth function $f : X^{(n)} \to \real$ such that $f{\comp}P(\mathbb{a})=1$ and $\supp{f} \subset O$.
Since $f{\comp}P$ is smooth, it is continuous and  $(f{\comp}P)^{-1}(0,\infty)$ is open in $U$ containing $\mathbb{a}$. In other words, $\mathbb{a}$ is an interior point of $P^{-1}(O)$.
Thus $P^{-1}(O)$ is open in $U$.
\end{proof}

\section{Regular Complex}

In this section, $X$ stands for a fat CW complex with skeleta $\{X^{(n)}\}$.

\begin{defn}
Let us assume that $X$ is equipped with attaching maps $h_{n}$, $n \ge 0$.
A fat CW complex $X$ is said to be \textit{regular}, if $h_{n}$ is an induction onto a $D$-open subset of $X^{(n-1)}$, $n \ge 0$.
In this case, we also say that $X$ is a regular CW complex.
\end{defn}

\begin{prop}\label{prop:hat_{h_{n}}}
If $X$ is regular, then, for each $n \!\ge\! 0$, the map $\hat{h}_{n} : \underset{j \in J_{n}}\coprod \mathbb{D}^{n,m_{n}(j)} \to X^{(n)}$ induced from the induction $h_{n}$ is also an induction 
onto a $D$-open subset of $X^{(n)}$.
\end{prop}
\begin{proof}
Since $h_{n}$ is an induction, the pushout $\hat{h}_{n}$ is naturally an induction.
So, we are left to show that $\Img{\hat{h}_{n}}$ is $D$-open.
Suppose $P : U \to X^{(n)}$ is a plot.
Then there is an open cover $\{U_{\alpha}\}$ of $U$ such that $P|_{U_{\alpha}}$ can be pulled back to a plot from $U_{\alpha}$ to either $X^{(n-1)}$ or $\mathbb{D}^{n,m_{n}(j)}$ for some $j \in J_{n}$, by the definition of $X^{(n)}$.
\par
In the first case, we may assume that $P|_{U_{\alpha}}$ is pulled back to a plot $P_{\alpha} : U_{\alpha} \to X^{(n-1)}$ as $P|_{U_{\alpha}} = \hat{i}_{n}{\comp}P_{\alpha}$.
Then $(P|_{U_{\alpha}})^{-1}(\Img{\hat{h}_{n}}) = P_{\alpha}^{-1}((\hat{i}_{n})^{-1}\Img{\hat{h}_{n}})) = P_{\alpha}^{-1}(\Img{\hat{h}_{n}} \cap X^{(n-1)}) = P_{\alpha}^{-1}(\Img{h_{n}}) \subset U_{\alpha}$. 
Since $\Img{h_{n}}$ is $D$-open, $P_{\alpha}^{-1}(\Img{h_{n}})$ is open in $U_{\alpha}$.
\par
In the second case, we may assume that $P|_{U_{\alpha}}$ is pulled back to a plot $P_{\alpha} : U_{\alpha} \to \mathbb{D}^{n,m_{n}(j)}$ for some $j \in J_{n}$ as $P|_{U_{\alpha}} = \hat{h}_{n}{\comp}P_{\alpha}$.
Thus, we have $P_{\alpha}^{-1}(\Img{\hat{h}_{n}}) = P_{\alpha}^{-1}(\mathbb{D}^{n,m_{n}(j)}) = U_{\alpha}$.
\par
In either case, $P|_{U_{\alpha}}(\Img{\hat{h}_{n}})$ is open in $U_{\alpha}$, and hence $\Img{\hat{h}_{n}})$ is $D$-open.
\end{proof}

If $X$ is regular, we can define its $D$-open subset $\mathring{X}$ of $X$ by induction as follows.
\begin{enumerate}
\item $\mathring{X}^{(-1)} = \emptyset$.
\item $\mathring{X}^{(n)} = \mathring{X}^{(n-1)} \cup \hat{h}_{n}(\underset{j \in J_{n}}\coprod \Int\mathbb{D}^{n,m_{n}(j)})$
\item $\mathring{X} = \bigcup_{n}\,\mathring{X}^{(n)}$
\end{enumerate}
Then we also have their boundary as $\partial{X} = X \smallsetminus \mathring{X}$ and $\partial{X}^{(n)} = X^{(n)} \smallsetminus \mathring{X}^{(n)}$, $n \ge 0$.

\section{Main Results}\label{sect:MainResult}

By J.~Watts \cite{WattsThesis}, a manifold with corners is a Fr\"olicher space, which can be embedded into \diffeology{} as a reflexive diffeological space.
Thus spheres and disks are all reflexive in \diffeology{}.
The following is our main result.

\begin{thm}\label{thm:main-ref}
A regular CW complex of finite dimension is reflexive.
\end{thm}
\begin{proof}
Let $X$ be a fat CW complex satisfying $X=X^{(n)}$, for some $n \!\ge\! 0$.
We show the statement by induction on $n \!\ge\! 0$.
In the case when $n=0$, we have nothing to do.
In the case when $n>0$, we may assume that we have done up to dimension $n-1$, and assume that $P : U \to X$ is a parametrization satisfying that, for every smooth function $f : X \to \real$, $f{\comp}P : U \to \real$ is a smooth function.
Since $X = X^{(n)}$, an open covering of $X$ is given by $\mathbb{D}^{n,m_{n}(j)}$ and $X^{n} \smallsetminus h_{n}(\underset{j \in J_{n}}{\textstyle\coprod}\mathbb{D}^{n,m_{n}(j)} \smallsetminus \mathbb{S}_{0}^{n-1,m_{n}(j)}) = X^{(n-1)} \smallsetminus h_{n}(\underset{j \in J_{n}}{\textstyle\coprod}\partial_{0}\mathbb{S}^{n-1,m_{n}(j)})$.
Then by Proposition \ref{prop:main-emsf}, we obtain that there is an open covering $\mathcal{U}=\{U_{\alpha}\}$ of $U$ such that $P_{\alpha}=P|_{U_{\alpha}} : U_{\alpha} \to X$ goes through either $X^{(n-1)}$ or $\mathbb{D}^{n,m_{n}(j)}$ for some $j \!\in\! J_{n}$.
\par
In case when $P_{\alpha}$ can be described as a composition $\hat{i}_{n}{\comp}P'_{\alpha} : U_{\alpha} \xrightarrow{P'_{\alpha}} X^{(n-1)} \xrightarrow{\hat{i}_{n}} X^{(n)}$, we have that $P'_{\alpha} : U_{\alpha} \rightarrow X^{(n-1)}$ is smooth:
for a given smooth function $f : X^{(n-1)} \to \real$, the composition $f{\comp}h_{n} : {\underset{j \in J_{n}}{\textstyle\coprod} \mathbb{S}^{n-1,m_{n}(j)}} \to \real$ is also a smooth function.
For each $j \!\in\! J_{n}$, $f_{\!j}=f{\comp}h_{n}|_{\mathbb{S}^{n-1,m_{n}(j)}} : \mathbb{S}^{n-1,m_{n}(j)} \to \real$ can be smoothly extendable on $\mathbb{D}^{n,m_{n}(j)}$, since $\mathbb{S}^{n-1,m_{n}(j)}$ is a closed subset of $\mathbb{D}^{n,m_{n}(j)}$ as a smooth function $\hat{f}_{\!j} : \mathbb{D}^{n,m_{n}(j)} \to \real$.
Since $\hat{f}_{\!j}$ and $f$ coincide with each other on $\mathbb{S}^{n-1,m_{n}(j)}$, they define a smooth function $\hat{f} : X^{(n)} \to \real$ whose restriction to $X^{(n-1)}$ is $f$.
Thus $\hat{f}{\comp}P_{\alpha}=f{\comp}P'_{\alpha} : U_{\alpha} \to \real$ is smooth.
Since $X^{(n-1)}$ is reflexive, $P_{\alpha} : U_{\alpha} \to X^{(n-1)}$ is a plot.
\par
In case when $P_{\alpha}$ can be described as a composition $\hat{h}_{n}{\comp}P'_{\alpha} : U_{\alpha} \xrightarrow{P'_{\alpha}} \mathbb{D}^{n,m_{n}(j_{0})} \xrightarrow{\hat{h}_{n}} X^{(n)}$ for some $j_{0} \!\in\! J_{n}$, we have that $P'_{\alpha} : U_{\alpha} \rightarrow \mathbb{D}^{n,m_{n}(j_{0})}$ is smooth:
for any $\mathbb{x} \in \mathbb{D}^{n,m_{n}(j_{0})}$, there is an open neighbourhood $O \subset \mathbb{D}^{n,m_{n}(j_{0})}$ of $\mathbb{x}$.
We choose a smaller open neighbourhood $V$ and $W$ of $\mathbb{x}$ such that $\mathbb{x} \in W \subset \Cl{W} \subset V \subset \Cl{V} \subset O$ where $\Cl{V}$ is a compact subset.
Let $W'=(P'_{\alpha})^{-1}(W)$ so that $\mathbb{x} \in P'_{\alpha}(W')$.
For any smooth function $f : O \to \real$, there is a smooth function $f' : O \to \real$ such that $f'|_{W}=f|_{W}$ and $\supp{f'} \subset \Cl{V}$.
Then we get a smooth function $\hat{f}' : \underset{j \in J_{n}}{\textstyle\coprod} \mathbb{D}^{n,m_{n}(j)} \to \real$ as its zero extension:%
\par\vskip1ex\noindent\hfil$\displaystyle
\hat{f}'(\mathbb{x}) = \begin{cases}
\,f'(\mathbb{x}),& \mathbb{x} \in O,
\\[.5ex]
\,0,& \mathbb{x} \not\in \supp f'.
\end{cases}
$\hfil\par\vskip1ex\noindent
Then $\hat{f}'{\comp}i_{n} : \underset{j \in J_{n}}{\textstyle\coprod} \mathbb{S}^{n-1,m_{n}(j)} \to \real$ has also a compact support which is closed in the open subset $\Img{h_{n}}$ in $X^{(n-1)}$, since $X$ is regular.
Hence $\hat{f}'{\comp}i_{n}$ has its zero-extension $\hat{f}_{0}$ on the entire $X^{(n-1)}$ so as to satisfy $\hat{f}'{\comp}i_{n} = \hat{f}_{0}{\comp}h_{n}$.
Thus smooth functions $\hat{f}'$ and $\hat{f}_{0}$ defines a smooth function $\hat{f} : X^{(n)} \to \real$ such that $\hat{f}{\comp}\hat{h}_{n}|_{W} = f$.
Thus $f{\comp}P'_{\alpha}|_{W'} = \hat{f}{\comp}\hat{h}_{n}|_{W}{\comp}P'_{\alpha}|_{W'}=\hat{f}{\comp}P_{\alpha}|_{W'}$ is smooth by the hypothesis, and hence $P'_{\alpha}|_{W'}$ is a plot.
Since $\mathbb{x} \in \mathbb{D}^{n,m_{n}(j_{0})}$ can be chosen arbitrary, $P'_{\alpha}$ is smooth, and so is $P_{\alpha}$.
Thus $P$ is a plot.
\end{proof}

In the above theorem, the finiteness condition on dimensions is essential, for the reflexivity is not preserved under taking colimits.

Now, let us recall \cite[Example 6.5]{MR4712607}.

\begin{expl}
The thin CW complex $\I$ is not reflexive.
\end{expl}

Let $n>0$, and let $\TD^{n} = \real^{n}/\pi^{n}_{set}$, where $\pi^{n}_{set} : \real^{n} \to \real^{n}$ is defined by 
\par\vskip1ex\noindent\hfil$\displaystyle
\pi^{n}_{set}(\mathbb{v}) = \left\{\,\begin{array}{ll}\mathbb{v}, \ & \ \Vert{\mathbb{v}}\Vert\le1\\[.5ex](\fracinline1/{\Vert{\mathbb{v}}\Vert}){\cdot}\mathbb{v}, \ & \ \Vert{\mathbb{v}}\Vert\ge1\end{array}\,\right\} \in D^{n} \subset \real^{n},
$\hfil\par\vskip1ex\noindent
which is a thin CW complex with one $0$-cell and one $n$-cell.
Then $\TD^{n}$ is topologically the same as an $n$-sphere $S^{n}$.
Similarly to \cite[Example 6.5]{MR4712607}, we obtain the following.

\begin{prop}\label{prop:main2}
The thin CW complex $\TD^{n}$ is not reflexive.
\end{prop}
\begin{Proof}
By definition, there is a canonical subduction $\pi^{n} : \real^{n} \twoheadrightarrow \real^{n}/\pi^{n}_{set}=\TD^{n}$. %
Let $\phi_{0} : (-1,1) \to \real$, $\psi_{0} : \real \to \real^{n}$ 
and $f : (-1,1) \to \TD^{n}$ 
be maps defined by
\par\vskip1ex\noindent\hfil$\displaystyle
f=\pi_{n}{\comp}\psi_{0}{\comp}\phi_{0},\qquad
\phi_{0}(t)=\textstyle\sqrt{\max\,\{\,0 \,, t\,\}},\qquad
\psi_{0}(t)=(1\!-\!t){\cdot}\mathbb{e},
$\hfil\par\vskip1ex\noindent
for a fixed vector $\mathbb{e}$ $=$ ${}^{t}(1,0,\dots,0) \in S^{n} \subset D^{n} \subset \real^{n}$.
Then we see that $f$ is not smooth at $t=0$.
In fact, if $f$ is smooth, $f$ can be expressed as $f$ $=$ $\pi_{n}{\comp}\phi$ near $t=0$ by a smooth map $\phi : \real \to \real^{n}$.
Then we have $\phi(t)$ $=$ $(1\!-\!\sqrt{t}){\cdot}\mathbb{e}$ for $t>0$, and hence $\phi'(0)$ $=$ $\underset{t \to -0}\lim \phi'(t)$ $=$ $(-\infty,0,\dots,0)$.
It contradicts to the smoothness of $\phi$ at $t=0$.
On the other hand, for any smooth function $g : \TD^{n} \to \real$, the composition $\psi$ $=$ $g{\comp}\pi_{n}{\comp}\psi_{0} : \real \to \real$ is also smooth on $\real$, and is constant on $(-\infty,0]$, as well.
Thus we have 
\par\vskip1ex\noindent\hfil$\displaystyle
\underset{t \to +0}\lim\,\psi^{(r)}(t)=\psi^{(r)}(0)=\underset{t \to -0}\lim\,\psi^{(r)}(t)=0\quad\text{for all \ \,$r \!\ge\! 1$.}
$\hfil\par\vskip1ex\noindent
By applying L'H\^opital's rule many times, one obtains that $\underset{t \to +0}\lim\,\psi^{(r)}(t)/t^{n}=0$ for all $r, \,n \ge 1$.
Then by induction, one can express $(\psi{\comp}f)^{(r)}(t)$ as the following form:
\par\vskip1ex\noindent\hfil$\displaystyle
(\psi{\comp}f)^{(r)}(t) = \underset{j=0}{\overset{r}{\sum}}\,P_{\!r,j}(1/\sqrt{t}){\cdot}\psi^{(j)}(\sqrt{t}), \,t>0,\quad\text{for all $r > 1$,}
$\hfil\par\vskip1ex\noindent
where $P_{\!r,j}(x)$ is a polynomial on $x$.
Again by applying L'H\^opital's rule, one obtains that $(\psi{\comp}f)^{(r)}(0)$ exists and equals to $\underset{t \to 0}\lim\,(\psi{\comp}f)^{(r)}(t)=0$ for all $r \!\ge\! 1$, and hence $\psi{\comp}f$ is smooth at $t\!=\!0$.
Thus $f \in \mathcal{D}'(\TD^{n})$ while $f \not\in \mathcal{D}(\TD^{n})$, where we denote $\mathcal{D}'(X)=\{\,P \in \mathcal{N}(X) \mid \text{$g{\comp}P$ is smooth for any smooth function $g : X \to \real$} \,\} \supset \mathcal{D}(X)$.
So, $\TD^{n}$ is not reflexive (outside $\Int{D^{n}} \subset \TD^{n}$).
In contrast, $\TD^{n}$ is reflexive at any point in $\Int{D^{n}}$.
\end{Proof}

\begin{cor}\label{cor:main2}
A thin CW complex of positive dimension is not reflexive.
\end{cor}

\begin{thm}\label{thm:main}
A closed manifold is a regular CW complex.
\end{thm}
\begin{proof}
In view of the standard Morse theory for a closed manifold (see Tamura \cite{MR841270}), it is sufficient to show that a smoothing of $N = M \cup_{h} D^{n} \times D^{m}$ a manifold with boundary obtained by attaching a handle $D^{n} \times D^{m}$ on $M$, where $h : S^{n-1} \times D^{m} \hookrightarrow M$ is a smooth embedding in \manifold{}, is diffeomorphic to $N' = M \cup_{h'} \mathbb{D}^{n,m}$ a manifold obtained by attaching a smooth handle $\mathbb{D}^{n,m}$ on $M$, where $h' : \mathbb{S}^{n-1,m} \to M$ is a diffeological embedding in {\diffeology}, so that we obtain the following pushout diagram in {\diffeology}.
\par\vskip1ex\noindent\hfil$\displaystyle
\begin{tikzcd}[ampersand replacement=\&]
{\mathbb{S}^{n-1,m}}
\arrow[r,"h'"]
\arrow[d,hook,"i'"]
\&
{M}
\arrow[d,hook,"\hat{i}'"]
\\[1ex]
{\mathbb{D}^{n,m}}
\arrow[r,"\hat{h}'"]
\&
{N'.\!\!}
\end{tikzcd}
$\hfil\par\vskip1ex\noindent

First, we take a collar neighbourhood $\partial{M}\times(-\fracinline1/2,\fracinline1/2] \subset M$ of $\partial{M}=\partial{M}\times\{\fracinline1/2\}$ in $M$.
Then we have a submanifold $M_{a} = M \smallsetminus \{\, (\mathbb{x},t) \in \partial{M}\times(-\fracinline1/2,\fracinline1/2] \mid t \!>\! a \,\}$ of $M$, diffeomorphic to $M$, for $a \in [-\fracinline1/4,\fracinline1/2)$.
We smoothly extend $h : S^{n-1} \times D^{m} \to \partial{M}$ to a diffeological embedding $h_{3} : S^{n-1} \times D_{3}^{m} \to \partial{M}$, where we denote $D_{r}^{d} = \{\,\mathbb{v} \in \real^{d} \mid \Vert{\mathbb{v}}\Vert \le r\,\}$ for $d \in \ordinal$ and $r > 0$, which can be obtained by using a diffeomorphism $D_{3}^{m} \homeo D^{m}$ if necessary.
For $a \in [-\fracinline1/4,\fracinline1/2)$ and $b \in [1,3]$, there is also a diffeological embedding $h_{a,b} : \partial{D_{1-a}^{n}} \times D_{b}^{m} = S^{n-1} \times D_{b}^{m} \to \partial{M} = \partial{M_{a}}$ as an extension of $h$, obtained by restricting $h_{3}$ to $S^{n-1} \times D_{b}^{m}$.
Then by definition, $h_{1}=h$.
Let $N_{a,b} = M_{a} \cup_{h_{a}} D_{1-a}^{n} \times D_{b}^{m}$, which is diffeomorphic to $N$ for $a \in [-\fracinline1/4,\fracinline1/2)$ and $b \in [1,3]$.
In the case when $b \le 2$, we have the following open neighbourhood of $\Img{h_{b}}$ in $N_{a,b+1}$:
\par\vskip1ex\noindent\hfil$\displaystyle
h_{b+1}(S^{n-1} \times \Int{D_{b+1}^{m}}) \times (a{-}\fracinline1/4,a] \cup_{h_{b}} D_{1-a}^{n} \times D_{b}^{m}
$\hfil\par\vskip1ex\noindent
which is diffeomorphic to $S^{n-1} \times (a-\fracinline1/4,a] \times \Int{D_{b+1}^{m}} \cup D_{1-a}^{n} \times D_{b}^{m}$ $\homeo$ $\mathbb{S}_{a}^{n-1} \times \real^{m} \cup D_{1-a}^{n} \times D_{b}^{m}$ $\subset$ $\real^{n} \times \real^{m}$, where we denote $\mathbb{S}_{a}^{n-1} = \real^{n} \smallsetminus \Int{D_{1-a}^{n}}$.
Thus $N_{a,b}$ has a nebula consisting of three manifolds with boundary, $M_{0} \smallsetminus \Img{h_{a}}$, $\mathbb{S}_{a}^{n-1} \times \real^{m} \cup (D_{1-a}^{n} \smallsetminus D_{\fracinlines1/2-a}^{n}) \times D_{b}^{m}$ and $\Int{D_{1-a}^{n}} \times D_{b}^{m}$ where the latter two are open subsets of $\mathbb{S}^{n-1} \times \real^{m} \cup D_{1-a}^{n} \times D_{b}^{m}$.
Hence we have the following pushout diagram in {\diffeology}:
\par\vskip1ex\noindent\hfil$\displaystyle
\begin{tikzcd}[ampersand replacement=\&]
{\mathbb{S}_{a}^{n-1} \times \real^{m}}
\arrow[r,"h_{a}"]
\arrow[d,hook,"i_{a}"]
\&
{M_{a}}
\arrow[d,hook,"\hat{i}_{a}"]
\\[1ex]
{\mathbb{S}_{a}^{n-1} \times \real^{m} \cup D_{1-a}^{n} \times D_{b}^{m}}
\arrow[r,"\hat{h}_{a,b}"]
\&
{N_{a,b}.\!\!\!\!}
\end{tikzcd}
$\hfil\par\vskip1ex\noindent

Let us denote by $N_{b}=N_{-\fracinlines1/4,b}$.
Then we may regard $M_{a}$ and $N_{b}$ as submanifolds of $N_{a,b}$, and hence we obtain $N_{a,b} = M_{a} \cup N_{b}$. 

Second, we define a smooth function $\lambda_{\epsilon}$ by $\lambda_{\epsilon}(t) = \lambda(\frac{t-\epsilon}{1-2\epsilon})$ for a fixed small $\epsilon\!>\!0$ \,($\epsilon < \fracinline1/{22}$ so that $\frac1{1-2\epsilon}<\fracinline{11}/{10}$), and hence we obtain $\lambda'_{\epsilon}(t) =\frac1{1-2\epsilon}{\cdot}\lambda'(\frac{t-\epsilon}{1-2\epsilon})$.
Hence by Proposition \ref{prop:maximumlambda'}, $\lambda'_{\epsilon}(t)$ has the maximum value $\frac{\ell(\fracinlines3/2)^{2}}{(1-2\epsilon){\cdot}\alpha} < 2$ when $\frac{t-\epsilon}{1-2\epsilon}=\fracinline1/2$, i.e, $t=\fracinline1/2$.
Let $f, \,g : \real \times \real \to \real$ be smooth functions defined for $(t,x) \in \real^{2}$ as follows:
\par\vskip1ex\noindent\hfil$\displaystyle
\begin{array}{l}
f(t,x) = x + t{\cdot}\lambda_{\epsilon}(4\!-\!2x),
\\[1ex]
g(t,x) = x - t{\cdot}\lambda_{\epsilon}(2x).
\end{array}
$\hfil\par\vskip1ex\noindent
Then, if $|t| \le \fracinline1/4$, we have $\pder{}by{x}f(t,x) = 1 - 2t{\cdot}\lambda'_{\epsilon}(4\!-\!2x) > 1-4|t| \ge 0$ and $\pder{}by{x}g(t,x) = 1 - 2t{\cdot}\lambda'_{\epsilon}(2x) > 1-4|t| \ge 0$.
Thus both $f$ and $g$ are strictly increasing on $x \in \real$.

Let $\widehat{N} = \partial{N_{0,1}} \times [-\fracinline1/4,\fracinline1/4]$.
Then, following \cite{MR841270}, we define a smooth map $\Psi : \widehat{N} \to N_{\fracinline1/4,\fracinline5/4}$ for $t \in [-\fracinline1/4,\fracinline1/4]$, 
$\mathbb{x} \in \partial{M} \smallsetminus h_{2}(S^{n-1} \times D_{2}^{m})$, 
$\mathbb{y} = h_{2}(\mathbb{u},\mathbb{v}),  \,(\mathbb{u},\mathbb{v}) \in S^{n-1} \times (\Int{D_{2}^{m}} \smallsetminus D^{m})$ and 
$\mathbb{z} = (\mathbb{u}',\mathbb{v}') \in D^{n} \times S^{m-1}$ as follows:
\begin{align*}&
\Psi(\mathbb{x},t) = (\mathbb{x},t) \in \partial{M} \times [-\fracinline1/4,\fracinline1/4] \subset M_{\fracinlines1/4}, 
\\[1ex]&
\Psi(\mathbb{y},t) = \begin{cases}\,
((1{-}t){\cdot}\mathbb{u},\mathbb{v}) \in (D^{n} \smallsetminus \Int{D_{\fracinlines7/8}^{n}}) \times \Int{D_{2}^{m}} \subset M_{\fracinlines1/4},&\mathbb{v} \in \Int{D_{2}^{m}} \smallsetminus D_{\fracinlines{(4-\epsilon)}/2}^{m},
\\[1.5ex]\,
((1{-}t){\cdot}\mathbb{u},\frac{f(t,\Vert{\mathbb{v}}\Vert)}{\Vert{\mathbb{v}}\Vert}{\cdot}\mathbb{v}) \in (D^{n} \smallsetminus \Int{D_{\fracinlines7/8}^{n}}) \times \Int{D_{2}^{m}} \subset M_{\fracinlines1/4},&\mathbb{v} \in \Int{D_{2}^{m}} \smallsetminus D_{\fracinlines3/2}^{m},
\\[1.5ex]\,
((1{-}t){\cdot}\mathbb{u},\mathbb{v}+\frac{t}{\Vert{\mathbb{v}}\Vert}{\cdot}\mathbb{v}) \in (D^{n} \smallsetminus \Int{D_{\fracinlines7/8}^{n}}) \times \Int{D_{2}^{m}} \subset M_{\fracinlines1/4},&\mathbb{v} \in \Int{D_{\fracinlines{(3+\epsilon)}/2}^{m}} \smallsetminus D^{m},
\end{cases}
\\[2ex]&
\Psi(\mathbb{z},t) = \begin{cases}\,
(\mathbb{u}'-\frac{t}{\Vert{\mathbb{u}'}\Vert}{\cdot}\mathbb{u}',(1{+}t){\cdot}\mathbb{v}') \in D^{n} \times \Int{D_{\fracinlines5/4}^{m}} \subset N_{\fracinlines5/4},&\mathbb{u}' \in D^{n} \smallsetminus D_{\fracinlines(1-\epsilon)/2}^{n},
\\[1.5ex]\,
(\frac{g(t,\Vert{\mathbb{u}'}\Vert)}{\Vert{\mathbb{u}'}\Vert}{\cdot}\mathbb{u}',(1{+}t){\cdot}\mathbb{v}') \in D^{n} \times \Int{D_{\fracinlines5/4}^{m}} \subset N_{\fracinlines5/4},&\mathbb{u}' \in \Int{D_{\fracinlines1/2}^{n}} \smallsetminus \{\mathbb{0}\},
\\[1.5ex]\,
(\mathbb{u}',(1{+}t){\cdot}\mathbb{v}') \in D^{n} \times \Int{D_{\fracinlines5/4}^{m}} \subset N_{\fracinlines5/4},&\mathbb{u}' \in \Int{D_{\fracinlines\epsilon/2}^{n}}.
\end{cases}
\end{align*}
Clearly, $\Psi$ is a smooth monomorphism into $M_{\fracinlines1/4} \cup N_{\fracinlines5/4} = N_{\fracinlines1/4,\fracinlines5/4}$, and so we regard $\widehat{N}$ as a subspace of $N_{\fracinlines1/4,\fracinlines5/4}$ by giving a pullback diffeology on it from $\Psi$.

Finally, a smooth function $\kappa : \partial{N} = \partial{M} \smallsetminus \Img{h} \cup D^{n} \times S^{m-1} \to [-\fracinline1/4,\fracinline1/4]$ is defined by $\kappa|_{\partial{M}} : \partial{M} \rightarrow \{0\} \subset [-\fracinline1/4,\fracinline1/4]$ and, for $\mathbb{x} = h_{2}(\mathbb{u},\mathbb{v}) \in h_{2}(S^{n-1} \times D_{2}^{m}) \supset h(S^{n-1} \times D^{m})$,
\par\vskip1ex\noindent\hfil$\displaystyle
\kappa(\mathbb{x})=\begin{cases}\,
\phi(1{-}\Vert{\mathbb{v}}\Vert) \le \phi(0),&(\mathbb{u},\mathbb{v}) \in \partial{M} \smallsetminus h(S^{n-1} \times \Int{D^{m}}) \subset S^{n-1} \times \real^{m},
\\[.5ex]\,
\phi(\Vert{\mathbb{u}}\Vert{-}1) \le \phi(0),&(\mathbb{u},\mathbb{v}) \in D^{n} \times S^{m-1}.
\end{cases}
$\hfil\par\vskip1ex\noindent
By the property \ref{prty:phi2}) of $\phi$, we have $\kappa(\mathbb{x}) = \phi(0) \in (0,\fracinline1/4)$ on $\mathbb{x} \in h(S^{n-1} \times S^{m-1})$. 
By definition, we have $\Img\kappa \subset (-\fracinline1/4,\fracinline1/4) \subset [-\fracinline1/4,\fracinline1/4]$.
Thus $N' = N_{\fracinlines3/4} \cup \{\,(\mathbb{x},t) \in \widehat{N} \smallsetminus N_{\fracinlines3/4} \mid t \le \kappa(\mathbb{x})\,\}$ gives a smoothing of $N_{\fracinlines3/4}$, which is nothing but a manifold with boundary obtained by attaching our smooth handle $\mathbb{D}^{n,m}$ to $M_{0}$ with $\mathbb{S}^{n-1,m}$ pasted into a neighbourhood of $\partial{M_{0}}$ in $M_{0}$.
\end{proof}

\begin{conj}
The converse of Theorem \ref{thm:main} is true, i.e., any regular CW complex $X$ with $X=X^{(n)}$ is a closed manifold.
\end{conj}

\begin{conj}
A manifold with boundary is a regular CW complex.
\end{conj}

\appendix

\section{Proof of Theorem \ref{thm:diffeotypeofhandle}}\label{appendix:diffeotypeofhandle}

Let $\alpha, \,\beta : [0,\infty) \times [-1,1] \to \real$ be smooth functions defined as follows:
\par\vskip1ex\noindent\hfil$\displaystyle
\alpha(u,v) = \fracinline3/2-p_{1}(u){\cdot}(1-\cos(\fracinline{\pi}/2{\cdot}v)),\quad%
\beta(u,v) = 1+\phi(\fracinline3/2{\cdot}u-1){\cdot}s(v),
$\hfil\par\vskip1ex\noindent
where smooth functions $p_{a}(x)$, $a>\fracinline1/2$, and an analytic function $s(x)$ are defined by
\par\vskip1ex\noindent\hfil$\displaystyle
p_{a}(x)=\begin{cases}\,\frac{\phi(\fracinlines3/2{\cdot}x-a)}{x}>0,&x>0\\[1.5ex]\,0,&x<\frac{(2a-1)}3\end{cases},\quad
s(x)=\begin{cases}\,\frac{\sin(\fracinlines{\pi}/2{\cdot}x)}{x}\ge0,&x\not=0\\[1.5ex]\,\frac{\pi}2>0,&x=0\end{cases}.
$\hfil\par\vskip1ex\noindent
By definition, we have $\beta(u,v) \ge 1$.
By Proposition \ref{prop:relation-phi}, we have $\phi(\fracinline3/2{\cdot}u-a) - \phi(a-\fracinline3/2{\cdot}u)=\fracinline3/2{\cdot}u-a$.
If $\fracinline3/2{\cdot}u \ge a+\fracinline1/2$, then $\phi(a-\fracinline3/2{\cdot}u)=0$ and $\phi(\fracinline3/2{\cdot}u-a)=\fracinline3/2{\cdot}u-a$, and hence $p_{a}(\fracinline3/2{\cdot}u)<\fracinline3/2$.
If $a \le \fracinline3/2{\cdot}u \le a+\fracinline1/2$, then $0 \le a-\fracinline3/2{\cdot}u+\fracinline1/2 \le \fracinline1/2$ and $0 \le \phi(a-\fracinline3/2{\cdot}u) \le \frac{(a-\fracinlines3/2{\cdot}u+\fracinlines1/2)^{2}}2 \le \frac{(a-\fracinlines3/2{\cdot}u+\fracinlines1/2)}4$.
Thus $\phi(\fracinline3/2{\cdot}u-a) \le \fracinline3/2{\cdot}u-a+\frac{(a-\fracinlines3/2{\cdot}u+\fracinlines1/2)}4< \fracinline9/8{\cdot}u$ and $p_{a}(\fracinline3/2{\cdot}u)<\fracinline9/8$.
If $a-\fracinline1/2 < \fracinline3/2{\cdot}u < a$, then $0 < \fracinline3/2{\cdot}u-a+\fracinline1/2 < \fracinline1/2$ and $\phi(\fracinline3/2{\cdot}u-a) \le \frac{(\fracinlines3/2{\cdot}u-a+\fracinlines1/2)^{2}}2<\frac{\fracinlines3/2{\cdot}u-a+\fracinlines1/2}4<\fracinline3/8{\cdot}u$ and $p_{a}(u)<\fracinline3/8$.
If $\fracinline3/2{\cdot}u\le a-\fracinline1/2$, then $p_{a}(\fracinline3/2{\cdot}u)=0$.
Thus always $p_{a}(\fracinline3/2{\cdot}u)<\fracinline3/2$ and $\alpha(u,v)>0$.

Using them, we define a smooth function $\Phi_{n,m} : \real^{n} \times D^{m} \to \real^{n} \times \real^{m}$ as follows:
\begin{align*}&\textstyle
\Phi_{n,m}(\mathbb{u},\mathbb{v}) = ({\alpha(\Vert{\mathbb{u}}\Vert,\Vert{\mathbb{v}}\Vert)}{\cdot}\mathbb{u},{\beta(\Vert{\mathbb{u}}\Vert,\Vert{\mathbb{v}}\Vert)}{\cdot}\mathbb{v}).
\end{align*}

Let $(x,y)=\Phi_{1,1}(u,v)$, $(u,v) \in \real \times [-1,1]$.
Then we have
\begin{align*}&
x = \fracinline3/2{\cdot}u-p_{1}(\fracinline3/2{\cdot}\vert{u}\vert){\cdot}(1{-}\cos(\fracinline{\pi}/2{\cdot}v)){\cdot}u,
\\&
y = v+\phi(\fracinline3/2{\cdot}\vert{u}\vert{-}1){\cdot}{\sin(\fracinline{\pi}/2{\cdot}v)}.
\end{align*}
\begin{lem}\label{lem:fathandle}
$\Phi_{1,1} : \real \times [-1,1] \to \real \times \real$ is a diffeomorphism onto $\mathbb{D}^{1,1} \subset \real \times \real$.
\end{lem}
\begin{proof}
If $v=0$, then we have $(x,y)=({u},{0})$.
If $\vert{v}\vert=1$, then we have $({x},{y})$ $=$ $((\fracinline3/2-p_{1}(\fracinline3/2{\cdot}\vert{u}\vert)){\cdot}{u},(1+\phi(\fracinline3/2{\cdot}\vert{u}\vert{-}1)){\cdot}v)$, and we have $\vert{x}\vert$ $=$ $\fracinline3/2{\cdot}\vert{u}\vert-\phi(\fracinline3/2{\cdot}\vert{u}\vert{-}1)$, $\vert{y}\vert$ $=$ $1+\phi(\fracinline3/2{\cdot}\vert{u}\vert{-}1)$ 
and $\vert{x}\vert+\vert{y}\vert = \fracinline3/2{\cdot}\vert{u}\vert+1$.
Hence we have $\vert{x}\vert$ $=$ $1-\phi(1{-}\fracinline3/2{\cdot}\vert{u}\vert)$ $=$ $1-\phi(2{-}\vert{x}\vert{-}\vert{y}\vert)$ which implies $({x},{y}) \in \partial\mathbb{D}^{1,1}$.

If $\vert{u}\vert \le \fracinline1/3$, then $({x},{y})=(\fracinline3/2{\cdot}{u},{v})$, while we obtain, in general, 
\begin{align*}
\vert{x}\vert+\vert{y}\vert &= \fracinline3/2{\cdot}\vert{u}\vert-\phi(\fracinline3/2{\cdot}\vert{u}\vert{-}1)(1-\cos(\fracinline{\pi}/2{\cdot}\vert{v}\vert)) + \vert{v}\vert + \phi(\fracinline3/2{\cdot}\vert{u}\vert{-}1)\sin(\fracinline{\pi}/2{\cdot}\vert{v}\vert) 
\\&
= \fracinline3/2{\cdot}\vert{u}\vert + \vert{v}\vert + \phi(\fracinline3/2{\cdot}\vert{u}\vert{-}1){\cdot}(\cos(\fracinline{\pi}/2{\cdot}\vert{v}\vert)+\sin(\fracinline{\pi}/2{\cdot}\vert{v}\vert)-1) 
\ge \fracinline3/2{\cdot}\vert{u}\vert + \vert{v}\vert, 
\end{align*}
since $\cos\theta+\sin\theta=\sqrt{2}\sin(\theta+\fracinline{\pi}/4) \ge 1$ if $0 \le \theta \le \fracinline{\pi}/2$.
Thus the image of $\Phi_{1,1}$ is $\mathbb{D}^{1,1}$.
\par
Now, let us calculate the Jacobian of $\Phi_{1,1}$.
In the case when $u > 0$, we have
\begin{align*}&
x = \fracinline3/2{\cdot}u-{\phi(\fracinline3/2{\cdot}{u}{-}1)}{\cdot}(1-\cos(\fracinline{\pi}/2{\cdot}v)),
\\&
y = v+\phi(\fracinline3/2{\cdot}u{-}1){\cdot}{\sin(\fracinline{\pi}/2{\cdot}v)},
\end{align*}
and then, it follows that $\pder{(x,y)}by{(u,v)}$ $=$ $\fracinline3/2{\cdot}(1-(1-\cos(\fracinline{\pi}/2{\cdot}v)){\cdot}\lambda(\fracinline3/2{\cdot}u-\fracinline1/2))$ $+$ 
$\frac{3\pi}4{\cdot}(\cos(\fracinline{\pi}/2{\cdot}v)+(1-\cos(\fracinline{\pi}/2{\cdot}v)){\cdot}\lambda(\fracinline3/2{\cdot}u{-}\fracinline1/2)){\cdot}\phi(\fracinline3/2{\cdot}u-1)$.
Hence, assuming $\pder{(x,y)}by{(u,v)}$ $=$ $0$, we obtain the following since $0 \le \cos(\fracinline{\pi}/2{\cdot}v), \,(1-\cos(\fracinline{\pi}/2{\cdot}v)){\cdot}\lambda(\fracinline3/2{\cdot}u{-}\fracinline1/2)$ $\le$ $1$ and $\phi(\fracinline3/2{\cdot}u{-}1)$ $\ge$ $0$: 
\par\vskip1ex\noindent\hfil$\displaystyle
(1-\cos(\fracinline{\pi}/2{\cdot}v))\cdot\lambda(\fracinline3/2{\cdot}u{-}\fracinline1/2)=1.
$\hfil\par\vskip1ex\noindent
It implies $\lambda(\fracinline3/2{\cdot}u{-}\fracinline1/2)=1$ and $\cos(\fracinline{\pi}/2{\cdot}v)=0$, and hence $\pder{(x,y)}by{(u,v)}$ $=$ $\frac{3\pi}4{\cdot}\phi(\fracinline3/2{\cdot}u-1)$.
Thus $\pder{(x,y)}by{(u,v)}$ $=$ $0$ implies $\phi(\fracinline3/2{\cdot}u-1)=0$.
The condition on $\lambda$ implies $u \ge 1$, while the condition on $\phi$ implies $u \le \fracinline1/3$, 
which is a contradiction, and we obtain $\pder{(x,y)}by{(u,v)} \not= 0$.
\par
In the case when $u < 0$, assuming $\pder{(x,y)}by{(u,v)} = 0$, we are led to a contradiction as well, by using an arguments parallel to the case when $u>0$.
\par
In the case when $-\fracinline1/3 < u < \fracinline1/3$, we have
$(x,y) = (\fracinline3/2{\cdot}u,v)$ and $\pder{(x,y)}by{(u,v)} = \fracinline3/2 \not= 0$.

Thus $\Phi_{1,1}$ has a smooth inverse function, and we have done.
\end{proof}

Now, we are ready to show Theorem \ref{thm:diffeotypeofhandle}.

Let $\Theta_{1,1} : \mathbb{D}^{1,1} \to \real \times D^{1}$ be the smooth inverse function of $\Phi_{1,1}$.
Then we obtain a smooth function $\Theta_{n,m}$ for $(\mathbb{x},\mathbb{y}) \in \mathbb{D}^{n,m}$ by the following formula:
\par\vskip1ex\noindent\hfil$\displaystyle
\Theta_{n,m}(\mathbb{x},\mathbb{y}) = (\frac1{\alpha(u,v)}{\cdot}\mathbb{x},\frac1{\beta(u,v)}{\cdot}\mathbb{y}),\quad (u,v)=\Theta_{1,1}(\Vert{\mathbb{x}}\Vert,\Vert{\mathbb{y}}\Vert).
$\hfil\par\vskip1ex\noindent
Then we have $\Psi_{1,1}(u,v)=(\Vert{\mathbb{x}}\Vert,\Vert{\mathbb{y}}\Vert)$, and hence $(\Vert{\mathbb{x}}\Vert,\Vert{\mathbb{y}}\Vert) = (\alpha(u,v){\cdot}u,\beta(u,v){\cdot}v)$.
Let $(\mathbb{u},\mathbb{v}) = \Theta_{n,m}(\mathbb{x},\mathbb{y})$.
Then we obtain 
$(\Vert{\mathbb{u}}\Vert,\Vert{\mathbb{v}}\Vert) = (\frac1{\alpha(u,v)}{\cdot}\Vert{\mathbb{x}}\Vert,\frac1{\beta(u,v)}{\cdot}\Vert{\mathbb{y}}\Vert) = (u,v)$ and 
$\Psi_{n,m}(\mathbb{u},\mathbb{v}) = (\alpha(\Vert{\mathbb{u}}\Vert,\Vert{\mathbb{v}}\Vert){\cdot}\mathbb{u},\beta(\Vert{\mathbb{u}}\Vert,\Vert{\mathbb{v}}\Vert){\cdot}\mathbb{v}) = (\alpha(u,v){\cdot}\mathbb{u},\beta(u,v){\cdot}\mathbb{v}) = (\mathbb{x},\mathbb{y})$.
Thus $\Theta_{n,m}$ is the inverse function of $\Phi_{n,m}$.
It completes the proof of the theorem.

\section{Proof of Theorem \ref{thm:smoothtypeofhandle}}\label{appendix:smoothtypeofhandle}

If we replace smooth function $\alpha$ in the previous section with the following $\widehat\alpha$, we must obtain a smooth bijection $\widehat\Phi_{n,m} : (\real^{n} \times D^{m},\mathbb{S}^{n-1} \times D^{m}) \to (\mathbb{D}^{n,m},\mathbb{S}^{n-1,m})$:
\begin{align*}&
\widehat\alpha(u,v)=\fracinline3/2-p_{1}(u)+p_{2}(u){\cdot}\cos(\fracinline{\pi}/2{\cdot}v),
\\&\textstyle
\widehat\Phi_{n,m}(\mathbb{u},\mathbb{v}) = ({\widehat\alpha(\Vert{\mathbb{u}}\Vert,\Vert{\mathbb{v}}\Vert)}{\cdot}\mathbb{u},{\beta(\Vert{\mathbb{u}}\Vert,\Vert{\mathbb{v}}\Vert)}{\cdot}\mathbb{v}).
\end{align*}
Let $(\mathbb{x},\mathbb{y}) = \widehat\Phi_{n,m}(\mathbb{u},\mathbb{v})$.
If $\Vert{\mathbb{u}}\Vert > 1$, then $\widehat\alpha(\Vert{\mathbb{u}}\Vert,\Vert{\mathbb{v}}\Vert){\cdot}\Vert{\mathbb{u}}\Vert = \fracinline3/2{\cdot}\Vert{\mathbb{u}}\Vert-\phi(\fracinline3/2{\cdot}\Vert{\mathbb{u}}\Vert{-}1) + \phi(\fracinline3/2{\cdot}\Vert{\mathbb{u}}\Vert{-}2)\cos(\fracinline{\pi}/2{\cdot}v) = 1 + \phi(\fracinline3/2{\cdot}\Vert{\mathbb{u}}\Vert{-}2)\cos(\fracinline{\pi}/2{\cdot}v) > 1$, and hence $\widehat\Phi_{n,m}(\mathbb{u},\mathbb{v}) \in \mathbb{S}^{n-1,m}$.
If $\Vert{\mathbb{u}}\Vert < 1$, then $\fracinline3/2{\cdot}\Vert{\mathbb{u}}\Vert-2 < -\fracinline1/2$, and hence the properties of $\phi$ imply
\begin{align*}
\widehat\alpha(\Vert{\mathbb{u}}\Vert,\Vert{\mathbb{v}}\Vert){\cdot}\Vert{\mathbb{u}}\Vert &= \fracinline3/2{\cdot}\Vert{\mathbb{u}}\Vert-\phi(\fracinline3/2{\cdot}\Vert{\mathbb{u}}\Vert{-}1) + \phi(\fracinline3/2{\cdot}\Vert{\mathbb{u}}\Vert{-}2)\cos(\fracinline{\pi}/2{\cdot}v) 
\\&
= 1 - \phi(1{-}\fracinline3/2{\cdot}\Vert{\mathbb{u}}\Vert) < 1,
\end{align*}
and hence $\widehat\Phi_{n,m}(\mathbb{u},\mathbb{v}) \not\in \mathbb{S}^{n-1,m}$.
If $\Vert{\mathbb{u}}\Vert = 1$, then $\widehat\alpha(\Vert{\mathbb{u}}\Vert,\Vert{\mathbb{v}}\Vert){\cdot}\Vert{\mathbb{u}}\Vert = 1$ and $\beta(\Vert{\mathbb{u}}\Vert,\Vert{\mathbb{v}}\Vert){\cdot}\Vert{\mathbb{v}}\Vert$ $=$ $\Vert{\mathbb{v}}\Vert+\phi(\fracinline1/2){\cdot}\sin(\fracinline{\pi}/2{\cdot}\Vert{\mathbb{v}}\Vert)$ $=$ $\Vert{\mathbb{v}}\Vert+\fracinline1/2{\cdot}\sin(\fracinline{\pi}/2{\cdot}\Vert{\mathbb{v}}\Vert)$ which ranges over $[0,\fracinline3/2]$, and hence we obtain $\widehat\Phi_{n,m}(S^{n-1} \times D^{m})=\partial_{0}\mathbb{S}^{n-1,m}$.

By arguments similar to that in Lemma \ref{lem:fathandle}, we have that $\widehat\Phi_{n,m}$ is diffeomorphic on $(\real^{n} \smallsetminus S^{n-1}) \times D^{m}$.
Now, let $(n,m)=(1,1)$.
Then, for $(x,y)=\widehat\Phi_{1,1}(u,v)$, we have
\begin{align*}&
x = \fracinline3/2{\cdot}u-\phi(\fracinline3/2{\cdot}|u|{-}1)+\phi(\fracinline3/2{\cdot}|u|{-}2){\cdot}\cos(\fracinline{\pi}/2{\cdot}v),
\\&
y = v+\phi(\fracinline3/2{\cdot}|u|{-}1){\cdot}{\sin(\fracinline{\pi}/2{\cdot}v)}.
\end{align*}
It then follows, by putting $u=1$, that 
\begin{align*}&
\pder{x}by{u}(1,v) = \frac32 - \frac32{\cdot}\lambda(1) +\frac32{\cdot}\lambda(0){\cdot}\cos(\fracinline{\pi}/2{\cdot}v)=0,
\\[1ex]&
\pder{x}by{v}(1,v) = -\frac{\pi}2{\cdot}\phi(-\fracinline1/2){\cdot}\sin(\fracinline{\pi}/2{\cdot}v)=0,
\end{align*}
which implies $\pder{(x,y)}by{(u,v)}(1,v)=0$, and $\widehat\Phi_{1,1}$ is not a diffeomorphism.
Further, we obtain that the Jacobian of $\widehat\Phi_{n,m}$ is zero on $S^{n} \times D^{m}$ for all $n, \,m \ge 1$.
Since the compact subspace $S^{n} \times D^{m} \subset \real^{n} \times D^{m}$ has a compact neighbourhood and $\mathbb{D}^{n,m}$ is a subspace of the euclidean space which is Hausdorff, $\widehat{\Phi}_{n,m}$ has continuous inverse near $\widehat{\Phi}_{n,m}(S^{n} \times D^{m})$, and hence $\widehat{\Phi}_{n,m}$ is a homeomorphism.
Details are left to the reader.

\section*{Acknowledgements}

This research is partly based on second author's master thesis \cite{Kojima:2023}, and is partially supported by Grant-in-Aids for Challenging Research (Exploratory) JP18K18713 and Scientific Research (C) JP23K03093 both from JSPS (Norio \textsc{Iwase}).

We are happy if the proof of our main theorem Theorem \ref{thm:main} could give an explicit illustration for the nature of a smooth handle decomposition of a closed manifold.

%
%


\bigskip

\end{document}